\documentclass[12pt]{amsart}
\usepackage[margin=1.05in]{geometry}
\usepackage{graphicx}
\usepackage{amsmath, amsthm, amsfonts, amssymb, mathtools, datetime, graphicx, url, longtable}
\usepackage[ruled]{algorithm2e}
\usepackage{fixfoot}
\usepackage[small,nohug,heads=vee]{diagrams}
\diagramstyle[labelstyle=\scriptstyle]
\usepackage{enumitem}
\usepackage{cancel}
\usepackage{systeme}
\usepackage{tikz-cd}
\tikzset{
  symbol/.style={
    draw=none,
    every to/.append style={
      edge node={node [sloped, allow upside down, auto=false]{$#1$}}
    },
  },
}

\usepackage{array, tabularx, float, multirow}

\usepackage[utf8]{inputenc}
\usepackage[english]{babel}
\usepackage[backend = biber,
    style = alphabetic,
    citestyle = alphabetic,
    maxbibnames=10,
    giveninits=true,
    doi=false,
    url=false
]{biblatex}
\usepackage{csquotes}
\DeclareFieldFormat{postnote}{#1}
\usepackage{multicol}
\usepackage{vwcol}

\usepackage{hyperref}
\hypersetup{
 colorlinks,
 linkcolor=blue,
 citecolor={blue}
 }
\usepackage[noabbrev,capitalize]{cleveref}

\theoremstyle{definition}

\newtheorem{Thrm}{Theorem}[section]
\newtheorem{Cor}[Thrm]{Corollary}

\newtheorem{Lem}[Thrm]{Lemma}
\newtheorem{Prop}[Thrm]{Proposition}

\newtheorem{Def}[Thrm]{Definition}

\newtheorem{Not}[Thrm]{Notation}
\newtheorem{Rem}[Thrm]{Remark}
\newtheorem{Constr}[Thrm]{Construction}

\newtheorem{conjecture}[Thrm]{Conjecture}

\newcommand{\C}{\mathbb{C}}
\newcommand{\F}{\mathbb{F}}

\newcommand{\PS}{\mathbb{P}}

\newcommand{\Z}{\mathbb{Z}}
\newcommand{\Q}{\mathbb{Q}}

\newcommand{\cE}{\mathcal{E}}

\newcommand{\cO}{\mathcal{O}}

\newcommand{\Id}{\mathrm{Id}}

\newcommand{\Bl}{\operatorname{Bl}}

\newcommand{\ch}{\operatorname{ch}}
\newcommand{\chn}{\operatorname{ch}}
\newcommand{\cn}{\operatorname{c}}
\newcommand{\codim}{\operatorname{codim}}

\newcommand{\PC}{\mathrm{PC}}
\newcommand{\ord}{\mathrm{ord}}
\newcommand{\rat}[0]{\operatorname{RatCurves}^n}

\newcommand{\CC}{\mathbb{C}}

\newcommand{\PP}{\mathbb{P}}
\newcommand{\QQ}{\mathbb{Q}}

\newcommand{\ZZ}{\mathbb{Z}}

\newcommand{\on}[1]{\operatorname{#1}}
\newcommand{\Span}{\on{Span}}
\newcommand{\Conv}{\on{Conv}}
\renewcommand{\NE}{\on{NE}}

\newcommand{\minRC}{m}

\newcounter{assumption}
\newcommand\assumption{\stepcounter{assumption}}

\begin{document}

\title[The minimal projective bundle dimension and toric $2$-Fano manifolds]{The minimal projective bundle dimension\\and toric $2$-Fano manifolds}

\makeatletter\let\@wraptoccontribs\wraptoccontribs\makeatother

\author[Araujo]{Carolina Araujo}
\address{Carolina Araujo, IMPA, 
Estrada Dona Castorina 110,
22460-320 Rio de Janeiro, Brazil}
\email{caraujo@impa.br}

\author[Beheshti]{Roya Beheshti}
\address{Roya Beheshti, Department of Mathematics \& Statistics, 
Washington University in St. Louis, St. Louis, MO, 63130}
\email{beheshti@wustl.edu}

\author[Castravet]{Ana-Maria Castravet}
\address{Ana-Maria Castravet, 
Universit\'e Paris-Saclay, UVSQ, Laboratoire de Math\'ematiques de Versailles, 
45 Avenue des \'Etats Unis, 78035 Versailles, France }
\email{ana-maria.castravet@uvsq.fr}

\author[Jabbusch]{Kelly Jabbusch}
\address{Kelly Jabbusch, Department of Mathematics \& Statistics, University of Michigan--Dearborn, 4901 Evergreen Rd, Dearborn, Michigan, 48128, USA}
\email{jabbusch@umich.edu}

\author[Makarova]{Svetlana Makarova}
\address{Svetlana Makarova, Department of Mathematics, University of Pennsylvania, 209 S 33rd St, Philadelphia, PA 19104, USA}
\email{murmuno@sas.upenn.edu}

\author[Mazzon]{Enrica Mazzon}
\address{Enrica Mazzon, Fakultät für Mathematik, Universität Regensburg, Universitäts\-strasse 31, 93040 Regensburg, Germany, and Department of Mathematics, University of Michigan, Ann Arbor, Michigan, 48109, USA}
\email{e.mazzon15@alumni.imperial.ac.uk}

\author[Viswanathan]{Nivedita Viswanathan}
\address{Nivedita Viswanathan, Department of Mathematical Sciences, Loughborough University, Loughborough, LE11 3TU, UK and School of Mathematical Sciences, University Park Campus, The University of Nottingham, Nottingham, NG7 2RD, UK}
\email{N.Viswanathan@lboro.ac.uk}

\contrib[with an appendix by]{Will Reynolds}
\address{Will Reynolds, School of Mathematics, The University of Edinburgh, Edinburgh, EH9 3FD, UK}
\email{W.R.N.Reynolds@sms.ed.ac.uk}

\date{}

\maketitle

\begin{abstract}
Motivated by the problem of classifying toric $2$-Fano manifolds, we introduce a new invariant for smooth projective toric varieties, the minimal projective bundle dimension. 
This invariant $\minRC(X)\in\{1, \dots,\dim(X)\}$ captures the minimal degree of a dominating family of rational curves on $X$ or, equivalently, the minimal length of a centrally symmetric primitive relation for the fan of $X$.
We classify smooth projective toric varieties with $\minRC(X)\geq \dim(X)-2$, and show that projective spaces are the only $2$-Fano manifolds among smooth projective toric varieties with 
$\minRC(X)\in\{1, \dim(X)-2,\dim(X)-1,\dim(X)\}$.
\end{abstract}

\tableofcontents

\section{Introduction}

Fano varieties are projective varieties with positive first Chern class. Over the complex numbers, this condition is equivalent to the existence of a metric with positive Ricci curvature. Basic examples of Fano varieties include projective spaces and Grassmannians.
The positivity condition has further geometric implications, e.g., Fano varieties over the complex numbers are \textit{simply connected}. This has an analogue on the algebro-geometric side:
any Fano variety is covered by rational curves \cite{mori79}, and is in fact {\em{rationally connected}}
\cite{KMM92, Campana}, i.e., there are rational curves connecting any two of its points.
In a series of papers, de Jong and Starr introduce and investigate possible candidates for the notion of \emph{higher rational connectedness} \cite{deJongHeStarr11, dJS07hF_and_rat_surfaces, dJS06note_on_positive_ch2, dJS06CI_are_RSC, Starr06hypersurfacesRSC}, inspired by the natural analogue in topology.
In particular, in \cite{dJS06note_on_positive_ch2} they define {\em $2$-Fano} manifolds. 
A smooth projective variety $X$ is $2$-Fano if it is Fano and its second Chern character 
$\ch_2(T_X)=\frac{1}{2}c_1(T_X)^2-c_2(T_X)$ is positive, i.e., $\ch_2(T_X)\cdot S>0$ for every surface $S$ in $X$. In a similar way, one can define $k$-Fano varieties for any $k\geq 2$, and  aim at their classification. 
For instance, $\PP^n$ is $n$-Fano, and it is conjectured that it is the only $n$-dimensional $n$-Fano manifold. The geometry of higher Fano manifolds has been fairly investigated, and in several special cases they are shown to enjoy the expected nice properties. 
For instance, $2$-Fano manifolds satisfying some mild assumptions are covered by rational surfaces \cite{dJS07hF_and_rat_surfaces}, and similar results hold for higher Fano manifolds \cite{Suzuki}, \cite{Nagaoka}.
There is a classification of $2$-Fano manifolds of high index
\cite{AraujoCastravet2013} and, more recently, a classification of homogeneous
$2$-Fano manifolds \cite{team2022}. 
On the other hand, very few examples of higher Fano manifolds are known. 
Quite strikingly, all known examples of $2$-Fano manifolds have Picard rank $1$ and relatively large index.

It is natural for algebraic geometers to turn to the pool of toric varieties when looking for intuition or examples. It is well known that projective spaces are the only projective toric manifolds with Picard rank $1$. Thus, a classification of toric $2$-Fano manifolds could either provide the first examples of $2$-Fano manifolds with higher Picard rank, or it could be an evidence that every $2$-Fano manifold has Picard rank $1$. 
Geometric properties of a toric variety can often be checked in the combinatorics of the associated fan. This bridge has been exploited in search of new examples of toric $2$-Fano manifolds
\cite{Nobili2011}, \cite{Nobili2012}, \cite{Sato2012}, \cite{Sato2016}, \cite{SatoSuyama2020}, \cite{SanoSatoSuyama2020},  \cite{Shrieve2020}.
Despite the efforts, a complete (computer aided) classification is only known up to dimension $8$ \cite{Nobili2011},  \cite{SanoSatoSuyama2020}, and projective spaces remain the only known examples of toric $2$-Fano manifolds. The sparsity of higher Fano manifolds leads to the following conjecture.

\begin{conjecture}\label{sparsity_conjecture}(\cite[Conjecture 4.3]{SanoSatoSuyama2020})
The only toric 2-Fano manifolds are projective spaces.
\end{conjecture}

In this paper, we propose a new strategy to approach Conjecture~\ref{sparsity_conjecture}.
We follow the philosophy introduced in \cite{AraujoCastravet2012}, namely, to investigate $2$-Fano manifolds by studying their \emph{minimal dominating families of rational curves}. 
By \cite{CFH}, minimal dominating families of rational curves on a smooth projective toric variety $X$
correspond to \textit{primitive relations} of the form
\begin{equation}\label{CSR}
 x_0 + \dots + x_{m}=0,
\end{equation}
satisfied by some of the primitive integral generators $x_i$ of the corresponding fan.
These primitive relations are called \emph{centrally symmetric of order $m+1$}.
By \cite{CFH}, a centrally symmetric primitive relation of order $m+1$
yields a $\PS^m$-bundle structure $X^{\circ}\to T$ on a dense open subset $X^{\circ}$ of $X$. 
If $\dim(T)\geq 1$, and the complement $X\setminus X^{\circ}$ has codimension at least $2$ in $X$, then one can construct a complete surface $S\subset X^{\circ}$ such that $\ch_2(T_X)\cdot S\leq 0$, showing that $X$ is not $2$-Fano. 
So our basic strategy consists of trying to describe, in a rather explicit way, a suitable birational map $\varphi:X\dashrightarrow Y$
transforming $X$ into a projective toric variety $Y$ admitting a $\PS^m$-bundle structure on a 
big open subset. 
We then hope to be able to compare 
the second Chern characters 
$\ch_2(T_X)$ and $\ch_2(T_Y)$ to show that  $X$ is not $2$-Fano,
except if $X=\PS^m$ and $\varphi$ is the identity. 

To follow this strategy, we introduce a new invariant of a smooth projective toric variety $X$, the \emph{minimal projective bundle dimension} of $X$, \emph{minimal $\mathbb{P}$-dimension} in short, which is of independent interest (\cref{def:m(x)}):
\[
\minRC(X) \ = \ \min\big\{ \ m\in \ZZ_{>0} \ \big| \ \text{ there is a relation as in (\ref{CSR}) } \big\} \ \in \ \{1, \dots, \dim X\}.
\]
By going through the database of toric Fano manifolds 
of low dimension and computing their primitive collections, one obtains Table~\ref{table_m}, indicating the number of Fano manifolds for each value of $\minRC$.  \cref{app} contains the code used to compute primitive collections. 

\begin{table}[h] \label{table_m}
\centering
\begin{tabular}{|l||l|l|l|l|l|l|l|}
\hline
\small{$\dim(X)$} & \small{\# Fanos}  & \small{\#($m$=1)} & \small{\#($m$=2)} & \small{\#($m$=3)} & \small{\#($m$=4)} & \small{\#($m$=5)} & \small{\#($m$=6)} \\ \hline \hline
4      & 124  & 107        & 15         & 1          & 1          &            &            \\ \hline
5      & 866  & 744        & 112        & 8          & 1          & 1          &            \\ \hline
6      & 7622 & 6333       & 1174       & 105        & 8          & 1          & 1         \\ \hline
\end{tabular}
\medskip
\caption{The minimal $\mathbb{P}$-dimension of toric Fano manifolds of low dimension.}
\end{table}

When $\minRC (X) = 1$, Casagrande  constructs in \cite{Cas03} a sequence of blowdowns and flips from $X$ to a toric variety admitting a global $\PP^1$-bundle structure. 
This allows us to make the basic strategy work, yielding the following result.

\begin{Thrm}\label{thm:m=1_not_2-Fano}
Let $X$ be a smooth toric Fano variety with $\minRC(X)=1$.
Then $X$ is not $2$-Fano.
\end{Thrm}

\noindent Table~\ref{table_m} suggests that this result covers ``most'' toric varieties, and not just the fringe cases.

Next we turn our attention to toric Fano manifolds $X$ with large values of $\minRC(X)$. 
Projective spaces are the only smooth projective toric varieties admitting a centrally symmetric primitive relation of order $\dim(X)+1$. 
In \cite[Proposition 3.8]{CFH}, Chen, Fu and Hwang classify toric Fano manifolds admitting a centrally symmetric primitive relation of order $\dim(X)$. There are three such varieties, and two of them also admit a centrally symmetric primitive relation of order $2$.
As a consequence, the only $n$-dimensional toric Fano  manifold $X$ with $m(X) = n-1$ is the blowup of $\PP^n$ along a linear $\PP^{n-2}$.
In \cite{BW}, Beheshti and Wormleighton investigate smooth projective toric varieties admitting a centrally symmetric primitive relation of order $\dim(X)-1$, showing that they have Picard rank $\rho(X) \leq 5$. Most of these varieties also admit centrally symmetric primitive relations of order $2$ or $3$, and we prove the following bound for the remaining ones. \cref{thm:classification_m>=n-2} shows that this bound is sharp.

\begin{Thrm}\label{theorem: P n-2 implies rho<=3}
    Let $X$ be a smooth toric Fano variety with $\dim(X)=n \geq 6$ and $\minRC(X) \geq 3$.
    If $X$ has a centrally symmetric primitive relation of order $n-1$,
    \[ x_0+x_1+ \cdots + x_{n-2}=0, \]
    then
    $\rho(X) \leq 3$. Moreover, $\minRC(X) = n-2$ and the above relation is the only centrally symmetric primitive relation of $X$.
\end{Thrm}

Using Theorem~\ref{theorem: P n-2 implies rho<=3} and Batyrev's description of smooth projective toric varieties with Picard rank $3$, we are able to classify $n$-dimensional smooth toric Fano varieties with $m(X) = n-2$. There are eight distinct isomorphism classes when $n \geq 6$, which can be explicitly described. The following statement summarizes the classification of toric Fano manifolds with $\minRC(X) \geq \dim(X)-2$. 

\begin{Thrm}\label{thm:classification_m>=n-2} 
    We have the following classification of smooth toric Fano varieties with $\minRC(X) \geq \dim(X)-2$.
    \begin{enumerate}
        \item The only $n$-dimensional smooth toric Fano variety $X$ with $\minRC(X) = n$ is $\PP^n$. 
        \item For $n \geq 3$, the only $n$-dimensional smooth toric Fano variety $X$ with $\minRC(X) = n-1$ is the blowup of $\PP^n$ along a linear $\PP^{n-2}$.
        \item For $n \geq 6$, there are eight distinct isomorphism classes of $n$-dimensional smooth toric Fano varieties $X$ with $\minRC(X) = n-2$. Namely:
            \begin{enumerate}
                \item  $X=\PS_S(\mathcal{E})$ is a $\PS^{n-2}$-bundle over a toric surface $S$, where 
                $(S,\mathcal{E})$ is one of the following: 
                    \begin{itemize}
                     \item $S=\PS^2$ and $\mathcal{E}= \mathcal{O}_{\PS^2}(1) \oplus \mathcal{O}_{\PS^2}^{\oplus n-2} $,
                     \item $S=\PS^2$ and  $\mathcal{E}= \mathcal{O}_{\PS^2}(1) \oplus \mathcal{O}_{\PS^2}(1) \oplus \mathcal{O}_{\PS^2}^{\oplus n-3} $,
                     \item $S=\PS^2$ and  $\mathcal{E}= \mathcal{O}_{\PS^2}(2) \oplus \mathcal{O}_{\PS^2}^{\oplus n-2} $,
                     \item $S=\PS^1 \times \PS^1$ and $\mathcal{E}= \mathcal{O}_{\PS^1 \times \PS^1}(1,1) \oplus \mathcal{O}_{\PS^1 \times \PS^1}^{\oplus n-2} $,
                     \item $S=\PS^1 \times \PS^1$ and $\mathcal{E}= \mathcal{O}_{\PS^1 \times \PS^1}(1,0) \oplus \mathcal{O}_{\PS^1 \times \PS^1}(0,1) \oplus \mathcal{O}_{\PS^1 \times \PS^1}^{\oplus n-3} $,
                     \item $S=\F_1$ and $\mathcal{E}= \mathcal{O}_{\F_1}(e+f) \oplus \mathcal{O}_{\F_1}^{\oplus n-2}$, where $e\subset \F_1$ is the $-1$-curve, and $f\subset \F_1$ is a fiber of $\F_1\to \PS^1$.
                    \end{itemize}
            In the first three cases, $\rho(X)=2$, while in the latter three cases, $\rho(X)=3$.    
                \item Let $Y\simeq \PS_{\PS^2}\big(\mathcal{O}_{\PS^2}(1) \oplus \mathcal{O}_{\PS^2}^{\oplus n-2}\big)$ be the blowup of $\PP^n$ along a linear subspace $L=\PP^{n-3}$, and denote by $E\subset Y$ the exceptional divisor. 
                Then $X$ is the blowup of $Y$ along a codimension $2$ center $Z\subset Y$, where:
                \begin{itemize}
                     \item $Z$ is the intersection of $E$ with the strict transform of a hyperplane of $\PP^n$ containing the linear subspace $L$, or
                     \item $Z$ is the intersection of the strict transforms of two hyperplanes of $\PP^n$, one containing the linear subspace $L$, and the other one not containing it. 
                \end{itemize}
            In both cases, $\rho(X)=3$.
            \end{enumerate}
    \end{enumerate}
\end{Thrm}

\begin{Cor} \label{cor:2Fano}
    The projective space $\PP^n$ is the only smooth $n$-dimensional toric $2$-Fano variety with $\minRC(X) 
    \in \{ 1, n-2, n-1, n \}$.
\end{Cor}

\medskip

This paper is organized as follows. In Section~\ref{section:background}, we review some results from toric geometry and fix notation. In particular, we discuss centrally symmetric primitive relations on  smooth projective toric varieties, describing explicitly their open subsets admitting a projective space bundle structure (Proposition~\ref{prop:Pk bundle}).
In Section~\ref{sec:m=1}, we study smooth toric Fano varieties with $\minRC(X)=1$, and prove Theorem~\ref{thm:m=1_not_2-Fano}.
In Section~\ref{sec:P n-2 implies rho<=3}, we investigate smooth projective $n$-dimensional toric varieties admitting a centrally symmetric primitive relation of order $n-1$, and prove Theorem~\ref{theorem: P n-2 implies rho<=3}.
In Section~\ref{sec:m=n-2}, we use this result, together with Batyrev's description of smooth projective toric varieties with Picard rank $3$, to prove Theorem~\ref{thm:classification_m>=n-2}. 
\medskip

\noindent {\bf{Acknowledgements.}} This collaboration started as a working group at the  ICERM ``Women in Algebraic Geometry Collaborative Research Workshop'' in July 2020.
A great part of this work was developed during the follow-up ``Collaborate@ICERM'' meeting in May 2022. We thank ICERM for the financial support and the great working conditions provided to us during our visit. 
We are very grateful to Cinzia Casagrande and Milena Hering for many enlightening discussions and explanations about  toric varieties, and to Will Reynolds for running his code and providing us with precious data about primitive relations of toric Fano manifolds. 

Carolina Araujo was partially supported by CAPES/COFECUB, CNPq and FAPERJ Research Fellowships. Roya Beheshti was supported by 
NSF grant DMS-2101935. 
Ana-Maria Castravet was partially supported by the ANR 20-CE40-0023 grant \emph{FanoHK}.
Enrica Mazzon was supported by the collaborative research center SFB 1085 \emph{Higher Invariants - Interactions between Arithmetic Geometry and Global Analysis} funded by the Deutsche Forschungsgemeinschaft. Nivedita Viswanathan was supported by the EPSRC New Horizons Grant No.EP/V048619/1.
 
%%%%%%%%%%%%%%%%%%%%%%%%%%%%%%%%%%%%%%%%%%%%%%%%%%%
%
%              SECTION 2
%
%%%%%%%%%%%%%%%%%%%%%%%%%%%%%%%%%%%%%%%%%%%%%%%%%%%

\section{Primitive collections} \label{section:background}

\subsection{Notation and background}
A toric variety is a normal complex variety $X$ that contains a torus $T=(\C^*)^n$ as a dense open subset, together with an action of $T$ on $X$ that extends the natural action of $T$ on itself. 
There is a one-to-one correspondence between $n$-dimensional toric varieties and \emph{fans} in $\Q^n$.
Let $N$ be a free abelian group of rank $n$, and consider the vector space $N_\Q = N \otimes_{\Z} \Q$. 
A \emph{fan}  in $N_\Q$ is a nonempty finite collection $\Sigma$ of strongly convex polyhedral cones in $N_\Q$ such that every face of a cone in $\Sigma$ is also a cone in $\Sigma$, and 
the intersection of two cones in $\Sigma$ is a face of each.
We write $\delta \prec \tau$ to express that $\delta$ is a face of $\tau$. 
One-dimensional cones in $\Sigma$ are called rays, and
each ray is generated by a primitive vector in $N$. The set of primitive vectors of $N$ generating rays of $\Sigma$ is denoted by $G(\Sigma)$.
We will write a cone $\tau \in \Sigma$ in terms of its primitive generators, $\tau= \langle v_1, \dots, v_l \rangle$, 
saying that the $v_i$'s generate $\tau$, and setting $G(\tau) \coloneqq \{v_1, \dots, v_l \} \subseteq G(\Sigma)$.

We denote by $X_\Sigma$ the toric variety corresponding to a fan $\Sigma$.
Conversely, given a toric variety $X$, we denote by $\Sigma_X$ the fan associated to $X$. 
There is a one-to-one inclusion-reversing correspondence between cones in $\Sigma$ and $T$-orbit closures in $X_\Sigma$.
Given a cone $\tau \in \Sigma$, we write $V(\tau) \subset X_\Sigma$ for the corresponding $T$-orbit closure, or $V(v_1,\dots,v_l)$ when $G(\tau) = \{ v_1,\dots,v_l\}$.
Note that $\dim(\tau)=\codim_{X_\Sigma}V(\tau)$.
We refer to \cite{Fulton} and \cite{CoxLittleSchenck2011} for the background on toric varieties.  

In this paper, we are mostly interested in \emph{smooth} and \emph{proper} toric varieties.
The smoothness conditions translates into the fan $\Sigma$ being \emph{regular}, i.e., for each cone $\tau\in \Sigma$, the set of generators $G(\tau)$ is part of a basis of $N$ (\cite[Definition~1.2.16]{CoxLittleSchenck2011}).
The properness condition translates into the fan $\Sigma$ being \emph{complete}, i.e., its support being the whole $N_\Q$. In what follows, smooth and proper toric varieties will be simply called \emph{toric manifolds}. We would like to classify toric $2$-Fano manifolds. 
Given a toric manifold $X$, there is an exact sequence (\cite[Theorem~8.1.1]{CoxLittleSchenck2011}):
$$
0\to \Omega_X^1 \to N^{\vee}\otimes_{\ZZ} \cO_X \to \bigoplus_{v\in G(\Sigma_X)}\cO_{V(v)}\to 0 \, ,
$$
from which one easily computes:
$$
c_1(X)\, = \sum_{v \in G(\Sigma_X)} V(v) \quad  \text{ and }  \quad
\chn_2(X)\, = \,\frac{1}{2} \sum_{v \in G(\Sigma_X)} V(v)^2.
$$

\begin{Def}(\cite[Definition 2.6]{Bat91})
Let $\Sigma$ be a regular complete fan in $N_\Q$.
A \emph{primitive collection} $P \subseteq G(\Sigma)$ is a nonempty set of primitive vectors of $N$ that does not generate a cone of $\Sigma$, but such that any proper subset of $P$ does. % generate a cone of $\Sigma$;
Equivalently, $P=\{v_1, \dots, v_r\} \subseteq G(\Sigma)$ is a primitive collection if and only if 
\[  \langle v_1, \dots , v_r \rangle \notin \Sigma \quad \quad \text{and} \quad \quad \langle v_1, \dots,\check{v_i}, \dots , v_r \rangle \in \Sigma \] 
for any $i=1,\dots, r$.
We denote by $\PC(\Sigma)$ the set of primitive collections of $\Sigma$. For a toric manifold $X$, we will talk about primitive collections of $X$ and write $\PC(X)$, meaning $\PC(\Sigma_X)$.
\end{Def}

\begin{Def}
Let $\Sigma$ be a regular complete fan in $N_\Q$.
Given a primitive collection
$P=\{v_1, \dots,v_r\} \in \PC(\Sigma)$, 
let $\sigma(P)= \langle w_1, \dots , w_s \rangle$ be the minimal cone  of $\Sigma$ such that $v_1 + \cdots + v_r \in \sigma(P)$. 
Then there is a relation
\[  r(P)\colon \,v_1 + \cdots + v_r = \mu_1 w_1 + \cdots + \mu_s w_s, \]
where $\mu_j \in \Z_{\geq 0}$ for $j=1, \dots, s$. We call  $r(P)$ the \emph{primitive relation} associated to $P$. We define the \emph{order} of $P$ as $\ord(P)=|P|=r$, while the \emph{degree} of $P$ as $\deg(P)=r- \sum_{j=1}^{s}\mu_j$.
\end{Def}

By \cite[Proposition 3.1]{Bat91}, for any primitive collection $P$, we have $P \cap \sigma(P) = \varnothing$. In particular, $\{v_1, \dots, v_r\} \cap \{w_1, \dots ,w_s\} = \varnothing$.

\begin{Def}\label{def:CCPC}
Let $\Sigma$ be a regular complete fan in $N_\Q$.
A primitive collection $P=\{x_0,\dots, x_k\}$ of $\Sigma$ is called \emph{centrally symmetric} if $\sigma(P)=\{0\}$, i.e. 
\[
    r(P) \colon  x_0 + \dots + x_k=0.
\]
\end{Def}

\begin{Lem}
\label{lemma: CSPR do not intersect}
   Let $\Sigma$ be a regular complete fan in $N_\Q$, and let $P$, $Q$ be two distinct centrally symmetric primitive collections. Then $P \cap Q = \varnothing$.
\end{Lem}
\begin{proof}
Write $r(P) \colon x_0 + \cdots + x_k = 0$ and
$r(Q) \colon y_0 + \cdots + y_l = 0$.
Assume that $P \cap Q \neq \varnothing$, then without loss of generality we may assume that $x_0 = y_0$.
But then subtracting this vector from both relations, we get
\[ x_1 + \cdots + x_k = y_1 + \cdots + y_l, \]
which shows that interiors of two distinct cones intersect. This is a contradiction.
\end{proof}

\begin{Lem}
\label{lemma: not too many CSPR}
   Let $\Sigma$ be a regular complete fan in $N_\Q$, and let $P$, $Q$ be two distinct centrally symmetric primitive collections. Then
   $\Span P \cap \Span Q = \{0\}$,
   in particular $|P| + |Q| - 2 \leq \dim N_\QQ$.
\end{Lem}
\begin{proof}
Write $r(P) \colon x_0 + \cdots + x_k = 0$ and
$r(Q) \colon y_0 + \cdots + y_l = 0$.
Take any vector $v \in \Span P \cap \Span Q$, so we can write it as 
\[
    v = \sum a_i x_i = \sum b_j y_j.
\]
By possibly adding $r(P)$ and $r(Q)$ to the sums, we can get that all $a_i, b_j \geq 0$, and up to relabelling the $a_i$, $b_j$, we can assume $a_0 = b_0 = 0$. But this shows that 
$v$ is in the intersection of two cones,
$\langle x_1, \dots, x_k \rangle \cap \langle y_1, \dots, y_l \rangle$, 
and the sets of generators are disjoint by \cref{lemma: CSPR do not intersect}, 
so $\langle x_1, \dots, x_k \rangle \cap \langle y_1, \dots, y_l \rangle = \{ 0 \}$ and
$v = 0$.

The last claim follows from considering the dimensions of $\Span P$ and $\Span Q$.
\end{proof}

Let $A_1(X_\Sigma)$ be the group of algebraic 1-cycles on $X_\Sigma$ modulo numerical equivalence, and set $\mathcal{N}_1(X_\Sigma)=A_1(X_\Sigma) \otimes_{\Z} \Q$. The Mori cone $\NE(X_\Sigma) \subset \mathcal{N}_1(X_\Sigma)$ is the cone generated by the classes of effective curves.
A primitive integral class generating an extremal ray of $\NE(X_\Sigma)$ is called an \emph{extremal class}.
There is an exact sequence:
\begin{diagram}
    0 
    & \rTo
    & A_1(X_\Sigma)
    & \rTo
    & \Z^{G(\Sigma)}
    & \rTo
    & N 
    & \rTo
    0 , \\
    & 
    & [C]
    & \rTo
    & \big( C \cdot V(v)\big)_{v \in G(\Sigma)}
    & 
    & 
    & 
    & 
    & 
     \\
    & 
    & 
    &
    & \big(\nu_v\big)_{v \in G(\Sigma)}
    & \rTo
    & \sum_{v \in G(\Sigma)} \nu_v v.
    & 
    &
    &
     \\
\end{diagram}
Thus the elements of $A_1(X_\Sigma)$ are identified with integral relations between the elements of $G(\Sigma)$. 
If the class $[C]$ corresponds to the relation $\sum_v \nu_v v =0$, then we have $-K_{X_\Sigma} \cdot C = \sum_v \nu_v$. 

\begin{Prop}(\cite[Lemma 1.4]{Cas03b}) \label{prop:effective classes}
Let $\Sigma$ be a regular complete fan in $N_\Q$.
A relation 
\[
\alpha_1 x_1 + \dots + \alpha_l x_l - \beta_1 y_1 - \dots - \beta_m y_m = 0,
\]
with $\alpha_i, \beta_j \in \Z_{>0}$, defines an effective class in $\mathcal{N}_1(X_\Sigma)$ provided that   $\langle y_1, \dots, y_m \rangle$ is a cone of $\Sigma$. 
\end{Prop}
We will usually write the above relation as 
\[
    \alpha_1 x_1 + \dots + \alpha_l x_l = \beta_1 y_1 + \dots + \beta_m y_m.
\]
It follows that primitive relations correspond to effective curve classes.
By abuse of notation, we will identify a primitive relation $r(P)$ with the corresponding curve class. Note that
$\deg(P)=-K_{X_\Sigma} \cdot r(P)$.
In the projective case we have the following description of $\NE (X_\Sigma)$.

\begin{Prop}(\cite[Theorem 2.15]{Bat91})
Let $\Sigma$ be a regular complete fan in $N_\Q$, and assume that 
$X_\Sigma$ is projective. Then the Mori cone is generated by primitive relations: 
\[\NE (X_\Sigma) = \sum_{P \in \PC(X_\Sigma)} \Q_{\geq 0}\,r(P).\]
\end{Prop}

\begin{Prop}(\cite[Theorem 2.4]{Reid1983})
\label{prop:extremal forms cones}
Let $\Sigma$ be a regular complete fan in $N_\Q$, and assume that  $X_\Sigma$ is projective. Let $\gamma$ be an extremal class in $\NE(X_\Sigma)$ whose corresponding primitive relation is 
\[
    r(P) \colon \, v_1 + \dots + v_r = \mu_1 w_1 + \dots +\mu_s w_s.
\]
Let $\tau= \langle z_1, \dots, z_l \rangle$ be a cone of $\Sigma$ such that 
$G(\tau) \cap P = G(\tau) \cap G(\sigma(P)) = \varnothing$,
and such that $\langle  \sigma(P), \tau \rangle = \langle w_1, \dots, w_s, z_1, \dots, z_l \rangle$ is a cone of $\Sigma$. Then, for each $i=1, \dots, r$,
\[ \langle P\setminus\{v_i\}, \sigma(P), \tau \rangle = \langle v_1, \dots, \check{v_i}, \dots, v_r,  w_1, \dots, w_s, z_1, \dots, z_l  \rangle \]
is also a cone of $\Sigma$.
\end{Prop}

\subsection{The minimal $\mathbb{P}$-dimension}\label{sec:CSPC}

Let $X$ be a toric manifold with regular complete fan $\Sigma_X$ in $N_\Q$. In this section, we discuss centrally symmetric primitive collections, introduced in \cref{def:CCPC}.

\begin{Prop}(\cite[Proposition 3.2]{Bat91})\label{prop:existence_of_CSR}
If $X$ is projective, then $\Sigma_X$ has a centrally symmetric primitive collection of order $k+1$
\begin{equation} \label{P=CSR}
    r(P) \colon \, x_0+ \dots + x_k=0
\end{equation}
for some $k\in\{1,\dots, \dim(X)\}$.
\end{Prop}

\begin{Def}\label{def:m(x)}
For a projective toric manifold $X$, we define the 
\emph{minimal $\mathbb{P}$-dimension} as
\begin{equation*}
\minRC(X) \coloneqq \min\Big\{ m\in \ZZ_{>0} \, \Big| \
\begin{gathered}
\text{$\Sigma_X$ has a centrally symmetric }\\ 
\text{ primitive collection of order $m+1$ }
\end{gathered}  \ \Big\}.
\end{equation*} 
\end{Def}

The next remark explains the terminology of \cref{def:m(x)} and highlights the significance of studying centrally symmetric primitive collections.

\begin{Rem}\label{rem:CFH}
In \cite{CFH}, Chen, Fu and Hwang provide a new geometric proof of \cref{prop:existence_of_CSR} by relating centrally symmetric primitive collections to \emph{minimal dominating families of rational curves}. 
We review some aspects of the theory of rational curves on varieties and refer to \cite{kollar96} for details. 
Given a smooth and proper uniruled variety $X$, there is a scheme $\rat(X)$ parametrizing rational curves on $X$. 
A \emph{dominating family of rational curves} on $X$ is an irreducible component of $\rat(X)$ parametrizing rational curves that sweep out a dense open subset of $X$. A dominating family of rational curves $H$ is said to be \emph{minimal} if, for a general point  $x\in X$, the subvariety of $H$ parametrizing curves through $x$ is proper. 
When $X$ is projective, there always exists a minimal dominating family of rational curves on $X$.
For instance, one can take $H$ to be a dominating family of rational curves on $X$ having minimal degree with respect to some fixed ample line bundle on $X$.
The minimal anti-canonical degree $l_X$ of a minimal dominating family of rational curves is a natural invariant of a Fano variety $X$. In \cite{Casagrande_Druel}, Casagrande and Druel investigate $n$-dimensional Fano varieties admitting a minimal dominating family of rational curves of anti-canonical degree $n$, and classify the cases when $l_X=n$. Currently, there are no general results when the  anti-canonical degree is $n-1$.

When $X=X_\Sigma$ is a toric variety, there is a  one-to-one correspondence between minimal dominating families of rational curves $H$ on $X$ and centrally symmetric primitive collections of $\Sigma$ (\cite[Proposition 3.2]{CFH}). 
Moreover, if the centrally symmetric primitive collection has order $k+1$ as in \cref{P=CSR} above, then there is a dense $T$-invariant open subset $U$ of $X$ and a $\PS^k$-bundle $\pi \colon U \to W$ such that the general curve parametrized by $H$ is a line on a general fiber $\pi$ (\cite[Corollary 2.6]{CFH}). 

It follows from this discussion that the minimal $\mathbb{P}$-dimension $\minRC(X)=l_X-1$ is the smallest integer $k$ such that $X$ admits a generic $\PS^k$-bundle structure. 
We have 
\[\minRC(X) \in \ \{1, \dots, n=\dim(X) \},\]
and $\minRC(X)=n$ if and only if $X \simeq \PS^n$.
By \cite[Proposition 3.8]{CFH}, there are three toric Fano manifolds admitting a centrally symmetric primitive relation of order $n=\dim(X)$, namely: $\PP^{n-1}\times \PP^1$, the blowup of $\PP^{n-1}\times \PP^1$ along a linear $\PP^{n-2}$, and the blowup of $\PP^n$ along a linear $\PP^{n-2}$.
The first two varieties also admit a generic $\PP^1$-bundle structure. 
As a consequence, the only $n$-dimensional toric Fano manifold $X$ with $m(X) = n-1$ is the blowup of $\PP^n$ along a linear $\PP^{n-2}$. In Section~\ref{sec:m=n-2}, we shall classify $n$-dimensional toric Fano manifolds $X$ with $m(X) = n-2$.
\end{Rem}

Let $P\in \PC(X)$ be a centrally symmetric primitive collection of order $k+1$. As explained in \cref{rem:CFH}, $P$ induces a $\PS^k$-bundle structure on a dense $T$-invariant open subset $U$ of $X$. In \cite[Corollary 2.6]{CFH}, the $T$-invariant open subset $U$ was taken as small as possible, namely, $U\cong \PS^k\times (\CC^*)^{n-k}$. For our purposes, we want to take $U$ as big as possible. So our next goal is to describe explicitly the biggest $T$-invariant open subset of $X$ on which $P$ induces a $\PS^k$-bundle structure.

\begin{Not} \label{not:no proj bundle}
Let $P=\{x_0,\dots, x_k\}\in \PC(X)$ be a centrally symmetric primitive collection. Denote by $\mathcal{E}_P$ the set of cones $\sigma = \langle v_1, \dots, v_r \rangle \in \Sigma_X$ such that $P \cap G(\sigma) = \varnothing$,
and $\{v_1, \dots, v_r, x_{j_1}, \dots, x_{j_s} \} \in \PC(X)$ for some $s \geq 1$, i.e., 
\[ \cE_P \coloneqq \left\{ \sigma \in \Sigma_X
\mid P \cap G(\sigma) = \varnothing \text{ and }
\exists P' \subsetneq P \text{ such that } P' \cup G(\sigma) \in \PC(X)
\right\}.
\]
We write
\[  V(\mathcal{E}_P) \coloneqq \bigcup_{\sigma \in \mathcal{E}_P} V(\sigma) \subset X. \]
\end{Not}

\begin{Prop} \label{prop:Pk bundle}
Let $P=\{x_0,\dots, x_k\}\in \PC(X)$ be a centrally symmetric primitive collection, and let 
$V(\mathcal{E}_P)$ be as in Notation~\ref{not:no proj bundle}.
Then the open subset $U= X \setminus V(\mathcal{E}_P)$ admits a $\PS^k$-bundle structure over a smooth toric variety.
\end{Prop}

In order to prove \cref{prop:Pk bundle}, we first prove two auxiliary lemmas.

\begin{Lem} \label{Lem:fanU}
Let $P=\{x_0,\dots, x_k\}\in \PC(X)$ be a centrally symmetric primitive collection, let 
$V(\mathcal{E}_P)$ be as in Notation~\ref{not:no proj bundle}, and set $U= X \setminus V(\mathcal{E}_P)$. Then the fan $\Sigma_U$ of $U$ consists of all cones of $\Sigma_X$ of the form
    \begin{equation} \label{equ:conesU}
    \tau'= \langle \tau, x_{j_1}, \dots, x_{j_m} \rangle,
    \end{equation}
    where $0 \leq m \leq k$,  and  $\tau \in \Sigma_X$ is such that 
    $\langle \tau,
    P\setminus\{x_i\} \rangle \in \Sigma$
    for every $i\in\{0,\dots, k\}$. 
    (When $m=0$, \cref{equ:conesU} means that $\tau'=\tau$.)
\end{Lem}

\begin{proof}
Recall that a cone $\sigma \in \Sigma_X$ corresponds to a $T$-orbit, which is dense and open in $V(\sigma)$. Hence, a cone $\sigma \in \Sigma_X$ is in $\Sigma_U$ if and only if the corresponding orbit does not intersect $V(\cE_P)$, which is equivalent to saying that $V(\sigma) \not\subseteq V(\cE_P)$.
It is immediate that the cones of the form (\ref{equ:conesU}) define a fan $\Sigma' \subset \Sigma_X$ in $N_\Q$, and $X_{\Sigma'}$ is a dense open subset of $X$. We now prove that the toric variety $X_{\Sigma'}$ coincides with $U$ by showing that a cone $\sigma \in \Sigma_X$ is of the form (\ref{equ:conesU}) if and only if $V(\sigma) \not\subseteq V(\mathcal{E}_P)$.

Consider $\sigma \in \Sigma_X \setminus \Sigma'$, which means that
$\langle G(\sigma) \cup P \setminus \{x_i\} \rangle \notin \Sigma_X$ for some $i$. Then the set $G(\sigma) \cup P \setminus \{x_i\}$ contains a primitive collection $S$, so the cone $\tau \coloneqq \langle S \setminus P \rangle$ is in $\cE_P$. But notice that $\tau \prec \sigma$, so $V(\sigma) \subseteq V(\tau) \subseteq V(\cE_P)$ and hence $\sigma$ is not in $\Sigma_U$.
Conversely, if $\sigma \in \Sigma_X \setminus \Sigma_U$, then $V(\sigma) \subseteq V(\cE_P)$, hence there exists $\tau \in \cE_P$ such that $V(\sigma) \subseteq V(\tau)$ and $G(\tau) \cup P' \in \PC(X)$ for some $P' \subset P$. Since $G(\tau) \subseteq G(\sigma)$, we conclude that $\langle G(\sigma) \cup P'\rangle \notin \Sigma_X$, i.e., $\sigma \not\in \Sigma'$.
\end{proof}

Consider the sequences
\begin{center}
\begin{tikzcd}[row sep=tiny]
0 \arrow[r]
& N_P \coloneqq \ker(\phi) \arrow[r]
& N \arrow[r,"\phi"]
& \overline N = N / \Z \langle x_0, \dots, x_k \rangle \arrow[r]
& 0,
\\
0 \arrow[r]
& (N_P)_\Q \arrow[r]
& N_\Q \arrow[r,"\phi_\Q"]
& \overline N _\Q \arrow[r]
& 0,
\\
& \Sigma_0 \arrow[r] 
& \Sigma_U \arrow[r]
& \overline \Sigma_U ,
&
\end{tikzcd}
\end{center}
where $\phi$ is the quotient map, the fan $\Sigma_0$ of $ (N_P)_\Q \simeq \Q^{k+1}$ is the subfan of $\Sigma_U$ of cones of the form (\ref{equ:conesU}) with $\tau=\{0\}$ (in particular note that $X_{\Sigma_0} \simeq \PS^k$), and $\overline\Sigma_U = \{\phi_\Q(\sigma) \,|\, \sigma \in \Sigma_U\}$.

\begin{Lem} Let the notation be as above. Then
$\overline\Sigma_U$ is a toric fan, and the linear map $\phi_\Q$ is compatible with the fans $\Sigma_U$ and $\overline \Sigma_U$.
\end{Lem}

\begin{proof}
The cones of $\overline \Sigma_U$ are exactly $\phi_\QQ(\tau)$ for $\tau \in \Sigma_U$ such that $G(\tau) \cap P = \varnothing$, so for simplicity we only consider these $\tau$.

\begin{itemize} 
\item[-] It is immediate that the cones of $\overline \Sigma_U$ are rational polyhedral, and that the faces of $\phi_\Q( \tau)$ are $\phi_\Q( \delta)$, for all subcones $\delta \prec \tau$. 

\item[-] We need to show that the cone $\phi_\Q( \tau)$ is strongly convex, i.e., if $y \in \phi_\Q( \tau)$ and $-y \in \phi_\Q( \tau)$, then $y=0$. This follows automatically from the fact that the images of the generators $\phi_\QQ(G(\tau)) = \{ \overline v_1 , \dots , \overline v_r \}$ are linearly independent, which we prove by contradiction. 
If they are linearly dependent,
then there exist $a_1, \dots, a_r\in \QQ$, not all $0$, such that $\Sigma_{i=1}^r a_i \overline v_i = 0$ in $\overline N_\Q$. This implies that there exist $b_j \in \Q$ for $j=1, \dots, k$ such that $\Sigma_{i=1}^r a_i v_i = \sum_{j=1}^{k} b_j x_j$, which is a contradiction since
$\langle G(\tau) \cup P \setminus \{x_0\} \rangle \in \Sigma$
and hence its primitive generators are linearly independent in $N_\Q$.

\item[-]  $\Sigma_U$ and $\overline \Sigma_U$ are compatible with $\phi_\Q$ as we have $\phi_\Q(\tau') \in \overline \Sigma_U$ for any $\tau' \in \Sigma_U$.
\end{itemize}
\end{proof}

\begin{proof}[Proof of \cref{prop:Pk bundle}]
Let the notation be as above.
Let $\hat \Sigma_U$ be the collection of fans of the form (\ref{equ:conesU}) above with $m=0$. 
It follows from the description of the cones of $\Sigma_U$ in \cref{Lem:fanU} that
\begin{enumerate}
\item $\phi_\Q$ maps each cone $\hat \tau \in \hat \Sigma_U$ bijectively to a cone $\overline \tau \in \overline \Sigma_U$ such that $\phi(\hat \tau \cap N) = \overline \tau \cap \overline N$.  Furthermore, the map $\hat \tau \mapsto \overline \tau$ defines a bijection $\hat \Sigma_U \to \overline \Sigma_U$;
\item given cones $\hat \tau \in \hat \Sigma_U$ and $\tau_0 \in \Sigma_0$, the sum $\hat \tau + \tau_0$ lies in $\Sigma_U$ and every cone of $\Sigma_U$ arises in this way.
\end{enumerate}
In the notation of \cite[Definition 3.3.18]{CoxLittleSchenck2011}, we say that 
\emph{$\Sigma_U$ is split by $\overline \Sigma_U$ and $\Sigma_0$}.
We conclude by \cite[Theorem 3.3.19]{CoxLittleSchenck2011} that $U = X \setminus V(\mathcal{E}_P)$ is a locally trivial fiber bundle over $X_{\overline \Sigma_U}$ with fiber $X_{\Sigma_0} \simeq \PS^k$.  
It follows automatically that $X_{\overline \Sigma_U}$ is smooth, since it is the base of a locally trivial fibration with a smooth total space.
\end{proof}

\subsection{Some properties of primitive collections}

Before focusing on toric Fano manifolds, we collect here two useful properties of primitive collections of arbitrary toric manifolds.
The first one, by Sato, describes the behaviour of primitive collections under a smooth toric blowdown. The second one, by Batyrev, describes primitive collections on toric manifolds of Picard rank $3$. 

\begin{Prop}(\cite[Corollary 4.9]{Sato2000})
\label{prop:PC_blowdown}
Let $X$ be a toric manifold, and let $f \colon X \rightarrow Y$ be the contraction associated to an extremal class in $\mathrm{NE}(X)$, corresponding to a primitive relation of the form 
\[r(Q)\colon t_1 + \dots+ t_s = z.\]
Then the fan $\Sigma_Y$ is obtained from $\Sigma_X$ by removing the ray generated by $z$, and
$X$ is the blowup of $Y$ along $V( t_1,\dots,t_s)$.
Furthermore, the primitive collections of $Y$ are precisely the following $P_Y\in \PC(Y)$:
\begin{itemize}
    \item $P_Y=P_X$ for some $P_X \in \PC(X)$ such that $z \notin P_X$ and $P_X \neq Q=\{t_1, \dots, t_s\}$;
    \item $P_Y=(P_X \setminus \{z\}) \cup \{t_1,\dots,t_r\}$ for some $P_X \in \PC(X)$ such that $z \in P_X$ and $(P_X \setminus \{z\}) \cup S \notin \PC(X)$ for any subset $S \subsetneq \{t_1,\dots,t_r\}$.
\end{itemize}
\end{Prop}

\begin{Prop}(\cite[Theorem 5.7, Theorem 6.6]{Bat91})
\label{thm:Batyrev result rho=3}
Let $X$ be a projective toric manifold with $\rho(X)=3$. Then the number of primitive collections of $\Sigma_X$ is either $l=3$ or $l=5$. Moreover, the set of generators $G(\Sigma_X)$ can be written as a disjoint union of $l$ nonempty subsets
\[ G(\Sigma_X) = X_0 \sqcup \cdots \sqcup X_{l-1}
\]
that define primitive collections and relations as follows:
\begin{itemize}
    \item Case $l=3$. Each $X_0$, $X_1$, $X_2$ is a primitive collection, and the corresponding primitive relations are extremal.
    \item Case $l=5$. There are five primitive collections of the form $X_i \sqcup X_{i+1}$, $0\leq i\leq 4$, where $X_5\coloneqq X_0$.
    To describe the primitive relations of $X$, we use the following notation. 
    We fix a labelling $(v_1, \dots, v_k)$ for the elements of $X_i$. If $\overline c=(c_1,\dots,c_k)\in \Z^k$, then $\overline  c\cdot X_i$ stands for 
    $c_1v_1+\dots +c_kv_k$. Moreover, we set $\overline 1= (1,\dots, 1)$.
    Then there are vectors $\overline c$ and $\overline b$ of nonnegative integers such that at least one entry in $\overline c$ is zero (up to relabelling, we may assume that $c_1=0$), and the primitive relations of $X$ are the following:
    \begin{align*}
        r_0 &\colon\quad
            \overline 1\cdot X_0 + \overline 1\cdot X_1 = \overline c\cdot X_2 + (\overline b+ \overline 1)\cdot X_3\\
        r_1 &\colon\quad
            \overline 1\cdot X_1 + \overline 1\cdot X_2 = \overline 1\cdot X_4,\\
        r_2 &\colon\quad
           \overline 1\cdot X_2 + \overline 1\cdot X_3 = 0,\\
        r_3 &\colon\quad
            \overline 1\cdot X_3 + \overline 1\cdot X_4 = \overline 1\cdot X_1,\\
        r_4 &\colon\quad
            \overline 1\cdot X_4 + \overline 1\cdot X_0 = \overline c\cdot X_2 + \overline b \cdot X_3.
    \end{align*}
    The relations $r_0$, $r_1$ and $r_3$ are extremal, while $r_2 = r_1 + r_3$ and $r_4 = r_0 + r_3$.
\end{itemize}
\end{Prop}

\begin{Rem}\label{rmk_rho<=3}
    In \cite[Corollary 1.2]{SatoSuyama2020}, Sato and Suyama use Proposition~\ref{thm:Batyrev result rho=3} to show that projective spaces are the only toric $2$-Fano manifolds with Picard rank $\rho\leq 3$.
\end{Rem}

\subsection{Primitive collections on toric Fano manifolds}

Let $X$ be a projective toric manifold with regular complete fan $\Sigma_X$  in $N_\Q$. 

\begin{Prop}(\cite[Proposition 2.3.6]{Bat99})
The toric variety $X$ is Fano  if and only if all primitive collections of $\Sigma_X$ have strictly positive degree.
\end{Prop}

\begin{Prop}(\cite[Corollary 4.4]{Cas03b}) \label{prop:deg1 extremal}
Assume that $X$ is Fano, and let $P \in \PC(X)$. If $\deg(P)=1$, then the corresponding curve class is extremal.
\end{Prop}

\begin{Prop} \label{prop:PC order 2}
Assume that $X$ is Fano, and let $x \in G(\Sigma_X)$.
\begin{enumerate}
    \item There is at most one primitive collection of order $2$ and degree $2$ containing $x$. If it exists, then it is of the form $x+ (-x) =0$, and $\minRC(X)=1$.
    \item (\cite[Lemma 3.3]{Cas03}) There are at most two primitive collections of order $2$ and degree $1$ containing $x$. If there are exactly two of them, then they are of the form $x+y=(-w)$ and $x+w=(-y)$, and $\minRC(X)=1$.
\end{enumerate}
\end{Prop}

\begin{Cor}
\label{lemma: one true enemy}
Assume that $X$ is Fano and $\minRC(X)>1$. Then any $x \in G(\Sigma_X)$ is contained in at most one primitive collection of order 2. If there is such a primitive collection, then it is of the form $x+y=z$.
\end{Cor}

\begin{Def}
\label{definition: enemy}
Let $x\in G(\Sigma_X)$. We say that $y \in G(\Sigma_X)$ is an \emph{opponent} of $x$ if $\langle x,y \rangle \notin \Sigma_X$. 
\end{Def}

\begin{Not}
Assume that $X$ is Fano and $\minRC(X)>1$. By \cref{lemma: one true enemy}, each vector $x\in G(\Sigma_X)$ has at most one opponent. If such an opponent exists, we denote it by $x'$.
\end{Not}

\begin{Lem}
\label{lemma: sum of enemies gives more enemies}
Assume that $X$ is Fano and $\minRC(X)>1$. Consider a pair of opponents $x,x' \in G(\Sigma_X)$. If there exist $y,z \in G(\Sigma_X)$ such that $x+x' = y+z$, then $z=y'$.
\end{Lem}

\begin{proof}
If $y=x$ or $y=x'$, the claim follows automatically. So we assume otherwise. Note that $\{ x, x'\} \in \PC(X)$, and $y$ and $z$ do not form a cone, as otherwise $x+x' = y+z$ would give us a primitive relation of degree $0$, which is impossible for a toric Fano manifold.
\end{proof}

\begin{Lem}\label{lemma: when is eff deg 1 PC}
Assume that $X$ is Fano and $\minRC(X)>1$. Assume that there exist $x,y,z,u,v$ in $G(\Sigma_X)$ such that $(*)$ $x+y+z=u+v$ and such that $\langle u,v\rangle\in\Sigma_X$. Then exactly one of the following must happen:
\begin{itemize}
\item[a. ] The vectors $x,y,z$ are pairwise distinct, and $\{x,y,z\}$ is a primitive collection with primitive relation $(*)$. In particular, the corresponding curve class is extremal.
\item[b. ] Up to relabeling, $v=z$, $y=x'$ and $x+x'=u$. 
\end{itemize}
\end{Lem}

\begin{proof}
Assume two of $\{x,y,z\}$ do not form a cone. 
For example, assume $x$, $y$ do not form a cone. Then $y=x'$, 
the opponent of $x$. Let $x+x'=\alpha$, for some $\alpha\in G(\Sigma_X)$. We have that $\alpha+z=u+v$. As $\langle u,v\rangle\in\Sigma_X$, \cref{prop:effective classes} implies that 
$\alpha+z=u+v$ corresponds to an effective class of degree $0$, which therefore implies that $\{\alpha,z\}=\{u,v\}$. Up to relabeling, we may assume $v=z$, and hence, $x+x'=u$ and we are in the situation b. 

Assume now that any two vectors in $\{x,y,z\}$ form a cone. Then $x,y,z$ are mutually disjoint and $\langle x,y,z\rangle\notin\Sigma_X$, as otherwise by \cref{prop:effective classes} we obtain an effective curve class  of degree $-1$, contradicting the fact that $X$ is Fano. 
It follows that $\{x,y,z\}$ is a primitive collection. Since $\langle u,v\rangle\in\Sigma_X$, it follows that $(*)$ is the associated primitive relation.  \cref{prop:deg1 extremal} now implies that the corresponding curve class is extremal. 
\end{proof}

\section{Toric Fano manifolds with $\minRC(X)=1$}
\label{sec:m=1}

In this section, we study toric Fano manifolds with $\minRC(X)=1$, and follow the strategy outlined in the introduction to show that they cannot be $2$-Fano. 

For any $x \in G(\Sigma_X)$, the set of primitive collections containing $x$ is denoted by
\[
\PC_x(X) = \{ P\in \PC(X) \ | \ x\in P\}.
\]

\begin{Prop}(\cite[Lemma 3.1]{Cas03}) 
\label{prop:decompositions m=1}
Assume that $X$ is a toric Fano manifold and that $P=\{x, -x \} \in \PC(X)$.
\begin{enumerate}
    \item Any $Q \in \PC_x(X) \setminus \{P\}$ has degree 1 (hence is extremal by \cref{prop:deg1 extremal}), and $r(Q)$ is of the form
    \[r(Q) \colon  x + {\underbrace{y_1 + \dots + y_h}_{\strut \scriptstyle \in \,\langle Q \setminus \{x\} \rangle}} = {\underbrace{z_1 + \dots + z_h}_{\strut \scriptstyle \in \sigma(Q)}},\]
    where we denote $\langle Q \setminus \{x\} \rangle \coloneqq \langle y_1, \dots, y_h \rangle$ and $\sigma(Q) \coloneqq \langle z_1, \dots, z_h \rangle$.
    
    \item For any $R \in \PC_x(X) \setminus \{P,Q\}$, we have
    \[ V( R \setminus \{x\} ) \cap V( Q \setminus \{x\} )= \varnothing  \quad \text{ and }\quad V(R \setminus \{x\}) \cap V(\sigma(Q))= \varnothing.\]

    \item For any $Q \in \PC_x(X) \setminus \{P\}$ with $r(Q) \colon  x + y_1 + \dots + y_h = z_1 + \dots + z_h$ we have $Q'=\{-x, z_1, \dots, z_h\} \in \PC_{-x}(X)$, $Q'$ has degree 1 (hence is extremal) and 
    \[ r(Q') \colon  -x + {\underbrace{z_1 + \dots + z_h}_{\strut \scriptstyle \in \langle Q' \setminus \{-x\} \rangle =\sigma(Q)}} ={\underbrace{y_1 + \dots + y_h}_{\strut \scriptstyle \in \sigma(Q')=\langle Q \setminus \{x\} \rangle }}. \]
\end{enumerate}
\end{Prop}

\begin{Cor}
Let $X$ be a toric manifold, and $P=\{x, -x \} \in \PC(X)$.
With \cref{not:no proj bundle},
\begin{equation}\label{equ:E_P for m=1}
\cE_P = \left\{ \langle Q \setminus \{v\} \rangle  \mid Q \in \PC_v(X) \setminus \{P\},  \quad  v = \pm x \right\}.
\end{equation}
If moreover $X$ is Fano, then $V(\mathcal{E}_P)$ has $0$, $2$ or $4$ components of codimension $1$ in $X$.
\end{Cor}

\begin{proof}
Let $X$ be a toric manifold, and $P=\{x, -x \} \in \PC(X)$.
The description of $V(\mathcal{E}_P)$ in \cref{equ:E_P for m=1} follows from
\cref{not:no proj bundle}.
In the Fano case, the number of components of codimension $1$ of $V(\mathcal{E}_P)$ equals the number of primitive collections of order $2$ and degree $1$ containing $x$ or $-x$, which is $0$, $2$ or $4$ by \cref{prop:PC order 2} and \cref{prop:decompositions m=1}.
\end{proof}

\begin{Prop} \label{prop:reduction codim 2}
Let $X$ be a toric Fano manifold,  
and $P=\{x, -x \} \in \PC(X)$. Then there exists a birational morphism
$f \colon X \rightarrow Y$ such that 
\begin{itemize}
    \item $P_Y \coloneqq \{x,-x\} \in \PC(Y)$,
    \item $V(\mathcal{E}_{P_Y})$ has codimension $\geq 2$ in $Y$,
    \item $f$ is a composition of at most two blow-downs with disjoint centers and smooth target:
    \begin{align} \label{eq:blowup loci}
    \nonumber \mathrm{Exc}(f) & = \bigcup_{\substack{Q \in \PC_x(X) \setminus \{P\} \,:\,\\  \ord(Q)=2}} V(\sigma(Q)) \, \subset X, \\
    f(\mathrm{Exc}(f))& = \bigcup_{\substack{Q \in \PC_x(X) \setminus \{P\} \,:\,\\ \ord(Q)=2}} V(Q) \, \subset Y.
    \end{align}
\end{itemize}
\end{Prop}

\begin{proof}
If $V(\mathcal{E}_P)$ has codimension $\geq 2$ in $X$, i.e., if $P$ is the unique primitive collection of order $2$ containing $x$, then the statement holds with $f=\Id$.
\medskip 

Assume now that $V(\mathcal{E}_P)$ has $2$ components of codimension $1$, i.e. we have primitive relations
\[  r(P) \colon x+ (-x)=0,  \quad \quad 
    r(Q_1) \colon x+y=z,  \quad \quad 
    r(Q_1') \colon -x+z=y, \]
and, for any other $R \in \PC_x(X) \cup \PC_{-x}(X)$, one has $\ord( R \setminus \{ \pm x\} )\geq 2$. 
Let $f_1 \colon X \rightarrow Y$ be the smooth blow-down induced by the extremal ray of $\mathrm{NE}(X)$ corresponding to $r(Q_1)$. By \cref{prop:PC_blowdown} and \cref{prop:decompositions m=1}, $\PC_x(Y) \cup \PC_{-x}(Y)$ consists of
\begin{itemize}
    \item[-] $P_{Y}=\{x,-x\}$;
    \item[-] $R_{Y}=R_X$ for some $R_X \in \PC_x(X) \cup \PC_{-x}(X) \setminus \{Q_1\}$ such that $z \notin R_X$. In particular we have $\ord(R_{Y} \setminus \{\pm x\}) \geq 2$;
    \item[-] $R_{Y}= (R_X \setminus \{z\}) \cup \{x,y\}$ for some $R_X \in \PC_z(X)$ such that $(R_X \setminus \{z\}) \cup \{x\} \notin \PC(X)$ and $(R_X \setminus \{z\}) \cup \{y\} \notin \PC(X)$. In particular, we have $ \langle R_{Y} \setminus \{x\} \rangle = 
    \langle  R_{X} \setminus \{z\} , y \rangle
    $, so $\ord( R_{Y} \setminus \{x\} ) \geq 2$.
\end{itemize}
It follows that $V(\mathcal{E}_{P_{Y}})$ has codimension $\geq 2$ in $Y$, so the proposition holds with $f=f_1$.
\medskip

Assume now that $V(\mathcal{E}_P)$ has $4$ components of codimension $1$, i.e., we have
\begin{align*}
r(P)\colon x+(-x)&=0            &
r(Q_1)\colon  x+y&=-w           &
r(Q_1')\colon -x+(-w)&=y        \\
y+ (-y)&=0                      &
r(Q_2)\colon x+w&=-y            &
r(Q_2')\colon -x+(-y)&=w        \\
w+ (-w)&=0                      &
w+y&=-x                         &
-y+(-w)&=x                      &
\end{align*}
by \cref{prop:PC order 2}.
By \cite[p.1487, Case (3)]{Cas03}, any other primitive collection $R \in \PC(X)$ is disjoint from $\{x,-x,y,-y,w,-w\}$. Let $f_1 \colon X \rightarrow X_1$ be the smooth blow-down induced by the extremal ray of $\mathrm{NE}(X)$ corresponding to $r(Q_1)$. By \cref{prop:PC_blowdown} the primitive collections of $X_1$ containing $x$ or $-x$ are only 
$$P_{X_1}=P \quad \quad (Q_2)_{X_1}=Q_2 \quad \quad (Q'_2)_{X_1}=Q'_2.$$ 
It follows that $r(Q_2)$ corresponds to an extremal curve class in $\mathrm{NE}(X_1)$. Let $f_2 \colon X_1 \rightarrow X_2$ be the smooth blow-down induced by the extremal ray of $\mathrm{NE}(X)$ corresponding to $r(Q_2)$. By \cref{prop:PC_blowdown} $P_{X_2}=\{x,-x\}$ is the only primitive collection in $\PC_x(X_2) \cup \PC_{-x}(X_2)$. This implies that $V(\mathcal{E}_{P_{X_2}}) = \varnothing$, so the proposition holds with $Y=X_2$ and $f=f_2 \circ f_1$. 
\end{proof}

\begin{Constr}\label{lem:construction S in Y}
Let $Y$ be a projective toric manifold of dimension $\geq 3$, and $P=\{x, -x \} \in \PC(Y)$ a centrally symmetric primitive collection of order $2$. Let $V(\mathcal{E}_{P})\subset Y$ be the closed subset defined in \cref{not:no proj bundle}, set $U:=Y \setminus V(\mathcal{E}_{P})$, and let $\pi: U \rightarrow W$ be the $\PS^1$-bundle given by \cref{prop:Pk bundle}.

\vspace{-3pt}
\begin{figure}[ht]
\includegraphics[width=0.5\textwidth]{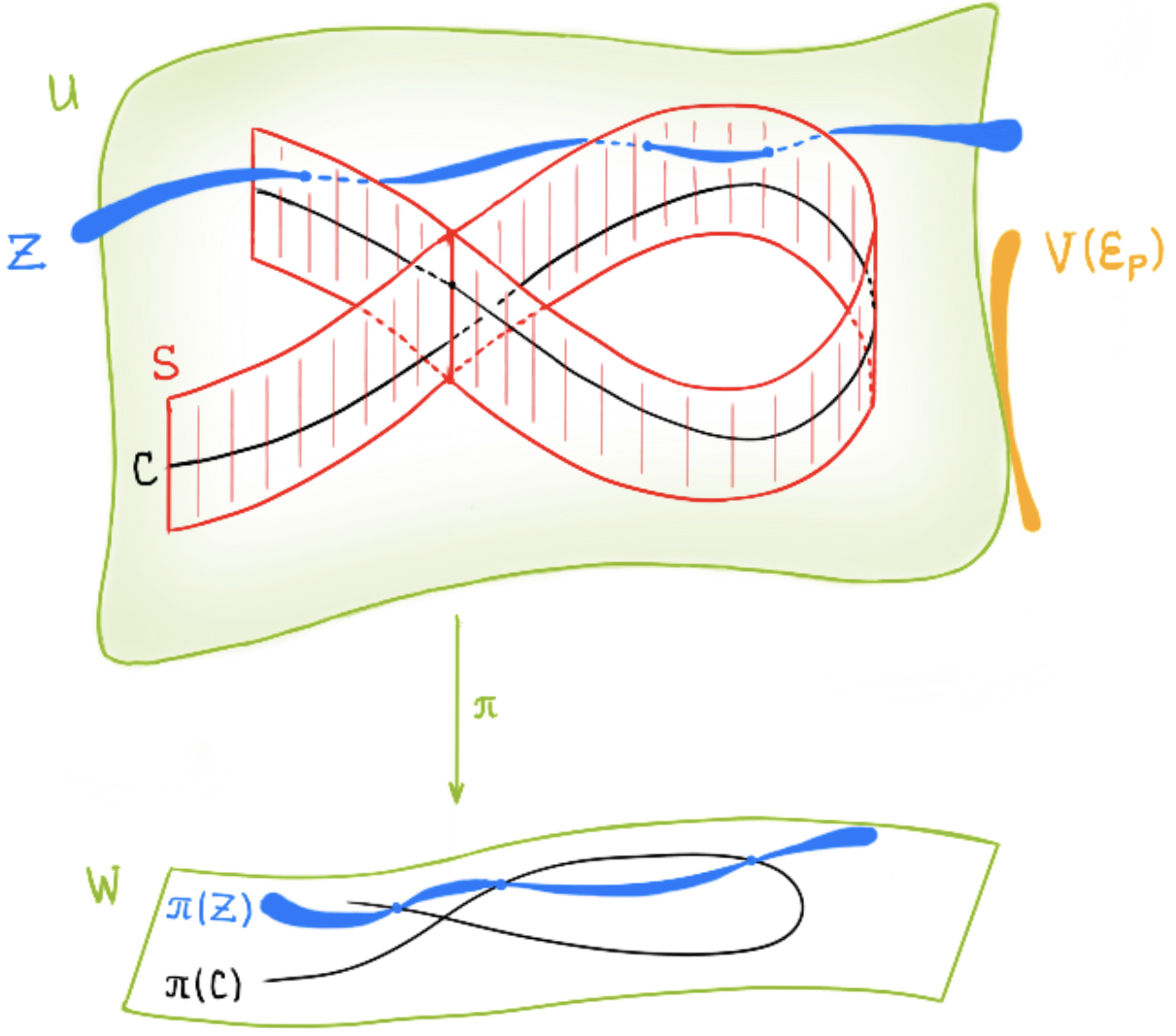}
\end{figure}

Assume that $V(\mathcal{E}_P)$ has codimension $\geq 2$ in $Y$, and let $Z \subset Y$ be any given closed subset of codimension $\geq 2$ in $Y$.
Note that a toric manifold is rational, hence rationally connected, so
by \cite[Proposition II.3.7]{kollar96}, there is a smooth (very free) rational curve $C \subset U \setminus Z$.
Consider the surface $S \coloneqq \pi^{-1}(\pi(C))\subset U$, and let $n\colon \tilde S\to S$ be its normalization. Then $\tilde S$ is a Hirzebruch surface with $\PS^1$-bundle structure $\tilde \pi:\tilde S\to \PS^1$ induced by $\pi$. The curve class of the image of a fiber of $\tilde \pi$ on $Y$   corresponds to the centrally symmetric relation $x+(-x)=0$.
By taking $C$ general, we may assume that $\pi(C)$ and $\pi(Z)$ meet transversely in at most finitely many general points.
Hence $S$ and $Z$ meet transversely in at most finitely many points. 
\end{Constr}

\begin{Prop} \label{thm:S.chY}
Let $Y$ be a projective toric manifold of dimension $\geq 3$, and $P_Y=\{x, -x \} \in \PC(Y)$ a centrally symmetric primitive collection of order $2$. Assume that $V(\mathcal{E}_P)$ has codimension $\geq 2$ in $Y$, and let $S \subset Y$ be as in \cref{lem:construction S in Y}. Then $S \cdot \chn_2(Y) = 0$.
\end{Prop}

\begin{proof}
Let $S = \pi^{-1}(\pi(C))$ be the surface from  \cref{lem:construction S in Y},
$n \colon \tilde S\to S$ its normalization, $\tilde \pi \colon \tilde S\to \PS^1$ the
$\PS^1$-bundle structure induced by $\pi$, and $F$ a fiber of $\tilde \pi$.
Our goal is to compute
\[
S \cdot \chn_2(Y) = \frac{1}{2} \sum_{v \in G(\Sigma_Y)} S \cdot V(v)^2
= \frac{1}{2} \sum_{v \in G(\Sigma_Y)}  \big(n^*V(v)\big)^2.
\]

Recall that the curve class of the image of $F$ in $Y$ is associated to the relation $x+(-x)=0$.
By restricting the divisors $n^*V(v)$ to $F$, we have:
\[
\begin{cases}
n^*V(v) \cdot F =0 & \text{ if } v \neq x,-x, \\
n^*V(x) \cdot F =1, &  \\
n^*V(-x) \cdot F =1. &  
\end{cases} 
\]
Hence there are sections $\sigma, \sigma'$ of $\tilde \pi \colon \tilde S\to \PS^1$, and $\alpha, \beta, \gamma \in \mathbb{Z}$ such that, on $\tilde S$,
\[
\begin{cases}
n^*V(v) =\alpha F 
& \text{ if } v \neq x,-x, \\
n^*V(x) =\sigma + \beta F, &  \\
n^*V(-x) =\sigma' + \gamma F. &  
\end{cases}
\]
Therefore 
\begin{align*}
\sum_{v \in G(\Sigma_Y)}  \big(n^*V(v)\big)^2
& = \sum_{v \neq x,-x} \cancel{\big(n^*V(v)\big)^2} + \big(n^*V(x)\big)^2 + \big(n^*V(-x)\big)^2 \\
& = \sigma^2 + 2\beta + \sigma'^2 + 2\gamma \\
& = \sigma^2 - 2(\sigma \cdot \sigma') + \sigma'^2  
\quad \text{ as }  \sigma \cdot \sigma' + \beta + \gamma 
= n^*V(x)\cdot n^*V(-x) = 0\\
& = (\sigma - \sigma')^2 =0 
\qquad \quad \hspace{5pt}
\text{ as $\sigma - \sigma'$ is a multiple of $F$},
\end{align*}
and this concludes the proof.
\end{proof}

\begin{Lem}(\cite[Lemma 5.1]{dJS06})
\label{Lem:ch_blowup}
Consider the blowup diagram
\begin{diagram}
E & \rInto^j & X \coloneqq \Bl_Z Y \\
\dTo^{\pi\coloneqq f_{|E}} && \dTo_f \\
Z & \rInto & Y
\end{diagram}
where both $Y$ and $Z$ are smooth projective varieties and $\codim_Y Z = c\geq 2$. 
Then we have the following relation between the 2nd Chern characters of $X$ and $Y$:
$$ \chn_2 (X) =
f^* \chn_2 (Y) + \frac{c+1}{2} E^2 - j_*\pi^* \cn_1(\mathcal{N}_{Z/Y})
.$$
\end{Lem}

\begin{Cor} \label{lem:blowdown_c2}
In the setting of Lemma \ref{Lem:ch_blowup}, assume that $c= 2$. Let $S \subset Y$ be a surface that intersects $Z$ at most transversely at $k\geq 0$ points, and let $S_X\subset X$ be its strict transform. Then
$$\chn_2(X) \cdot S_X =
\chn_2(Y) \cdot S - \frac{3}{2} \cdot k.$$
\end{Cor}

\begin{proof}
Write $S \cap Z = \{p_1,\dots, p_k\}$. Then $S_X$ is isomorphic to the blowup of $S$ at $p_1,\dots, p_k$, $S_X \cap E= \cup_{i=1}^{k}e_i$, where $e_i \simeq \PS^1$ is the exceptional curve over $p_i$, and $(e_i^2)_{S_X}=-1$.
By \cref{Lem:ch_blowup}, 
\begin{align*}
    \chn_2(X) \cdot S_X 
    & = f^* \chn_2 (Y) \cdot S_X + \frac{c+1}{2} E^2 \cdot S_X - j_*\pi^* \cn_1(\mathcal{N}_{Z/Y}) \cdot S_X \\
    & = \chn_2 (Y) \cdot  S + \frac{3}{2} (E_{|S_X})^2  - \pi^* \cn_1(\mathcal{N}_{Z/Y}) \cdot {S_X}_{|E} \\
    & = \chn_2 (Y) \cdot  S + \frac{3}{2} \sum_{i=1}^{k}(e_i)^2  - \pi^* \cn_1(\mathcal{N}_{Z/Y}) \cdot \sum_{i=1}^{k} e_i \\
    & = \ch_2 (Y) \cdot  S - \frac{3}{2}k - \sum_{i=1}^{k} \cancel{\cn_1(\mathcal{N}_{Z/Y}) \cdot  p_i}.
\end{align*} 
\end{proof}

We are ready to prove \cref{thm:m=1_not_2-Fano}: a toric Fano manifold $X$ with $\minRC(X)=1$ is not $2$-Fano.

\begin{proof}[{Proof of \cref{thm:m=1_not_2-Fano}}]
Let $X$ be a toric Fano manifold with $\minRC(X)=1$, and fix a centrally symmetric primitive relation $r(P) \colon x + (-x) = 0$.
Let $f \colon X \rightarrow Y$ be as in \cref{prop:reduction codim 2}, and $\pi \colon U=Y \setminus V(\mathcal{E}_{P_Y}) \rightarrow W$ the $\PS^1$-bundle structure induced by $r(P_Y) \colon x+ (-x)=0$ (see \cref{prop:Pk bundle}). By \cref{eq:blowup loci}, $Z \coloneqq f(\mathrm{Exc}(f))$ is empty or has codimension $2$ in $Y$. Let $S \subset Y$ be the surface given by \cref{lem:construction S in Y}. Then $S$ and $Z$ meet transversely in at most finitely many points. It follows from \cref{thm:S.chY} that $\ch_2(Y) \cdot S=0$. Let $S_X$ be the strict transform of $S$ in $X$. By \cref{lem:blowdown_c2}, $\ch_2(X) \cdot S_X\leq \ch_2(Y) \cdot S = 0$, and so $X$ is not $2$-Fano.
\end{proof}

\section{Proof of \cref{theorem: P n-2 implies rho<=3}}
\label{sec:P n-2 implies rho<=3}

In this section, we work in the setting of \cref{theorem: P n-2 implies rho<=3}:
\begin{enumerate}[label=(\Roman*)]
\assumption %now \theassumption=1
    \item\assumption \label{item: assumption on X, dim X}
    $X$ is a Fano manifold of dimension $n \geq 6$ with fan $\Sigma= \Sigma_X$ in  $N_\Q$;
    \item\assumption \label{item: m >= 3}
    $\minRC(X) \geq 3$;
    \item\assumption \label{item: CSPR of length n-1}
    $r(P)\colon x_0 + \cdots + x_{n-2} = 0$ is a centrally symmetric primitive relation in $\Sigma$.
\end{enumerate}
The equality $\minRC(X) = n-2$, as well as uniqueness of the centrally symmetric primitive collection, follows immediately from \cref{lemma: not too many CSPR}.
Now the goal is to prove that $\rho(X) \leq 3$, and this will follow from
\cref{proposition: (a) x+y+z = xi+xj}, \cref{proposition: (b) x+y+z = v} and \cref{prop: none PC}. By \cref{rmk_rho<=3}, we conclude that the projective space $\PP^n$ is the only smooth $n$-dimensional toric $2$-Fano variety with $\minRC (X) \geq n-2$ (\cref{cor:2Fano}).
\medskip
 
Let $P \coloneqq \{ x_0 , \dots , x_{n-2} \}$ and set $\Gamma = \Span P \subset N_{\QQ}$. 
Consider the quotient
\[ \pi \colon N_{\QQ} \simeq \QQ^n \to \QQ^n / \Gamma \simeq \QQ^2 .\]
Since $X$ is complete, the support of $\Sigma$ is equal to $N_{\mathbb{Q}}$, and hence we can find generators 
\begin{enumerate}[start=\theassumption,label=(\Roman*)]
    \item\assumption \label{item: x y z not in S}
    $x, y, z \in G(\Sigma) \setminus P$, for which
    \item\assumption \label{item: 0 in convex hull of x y z}
    $0 \in \Conv (\pi(x), \pi(y), \pi(z))$.
\end{enumerate}

\begin{Lem}
\label{lemma: x y z not a cone}
    The nonnegative span $\langle x, y, z \rangle$ is not a cone in $\Sigma$. 
\end{Lem}

\begin{proof}
    We will argue by contradiction and assume that $\langle x, y, z \rangle \in \Sigma$.
    By \ref{item: 0 in convex hull of x y z}, we can find a nonnegative triple of constants $(c_1,c_2,c_3)$ such that
    $c_1 \pi(x) + c_2 \pi(y) + c_3 \pi(z) = 0$,
    or in other words
    \[ v \coloneqq c_1 x + c_2 y + c_3 z = a_0 x_0 + \cdots + a_{n-2} x_{n-2} 
    \]
    for some constants $a_i$.
    Note that by adding some multiple of $r(P)$ \ref{item: CSPR of length n-1} to the right hand side, we can assume the $a_i$'s are nonnegative and such that at least one of them, say $a_j$, is $0$. It follows that $v$ lies in two cones of $\Sigma$, namely
    \[v \in \langle x,y,z \rangle \cap \langle x_0, \dots, \check{x_j}, \dots,x_{n-2}\rangle,
    \]
    which is impossible since $\{x,y,z\} \cap S=\varnothing$.
\end{proof}

Since $\{ x, y, z \}$ does not span a cone, we conclude that 
\begin{enumerate}
    \item either it is a primitive collection,
    \item or two of these vectors do not form a cone.
\end{enumerate}
The former case is the more technical one, and we start with it in \cref{subsection: x y z form a PC}. After we are done analyzing it, we can assume that none of the triples $\{ x, y, z \}$ as in \ref{item: x y z not in S} and \ref{item: 0 in convex hull of x y z} form a primitive collection, and this case will be treated in \cref{subsection: none of x y z form a PC}.

\subsection{\texorpdfstring{First case: $\{ x, y, z \}$}{\{x, y, z\}} is a primitive collection}
\label{subsection: x y z form a PC}

We recall that if $X$ is a projective toric Fano manifold and $\minRC (X)>1$, then any $x \in G(\Sigma)$ has at most one opponent by \cref{lemma: one true enemy}, where the opponent of $x$ is an element $x' \in G(\Sigma)$ such that $\{x,x'\} \in \PC(X)$. When we write $\{x, x'\}$, we mean either the set of two elements if $x'$ exists, or the singleton $\{x\}$ if $x'$ does not exist.

\begin{Rem}
\label{remark: plane geometry}
In the setting \ref{item: assumption on X, dim X}---\ref{item: 0 in convex hull of x y z},
pick a generator $u \in G(\Sigma)$. Then \ref{item: 0 in convex hull of x y z}  and plane geometry imply that the convex hull of $\pi(u)$ together with two of the vectors $\pi(x)$, $\pi(y)$, $\pi(z)$ contains $0$.
\end{Rem}

\begin{Lem}
\label{lemma: PR of x y z}
In the setting \ref{item: assumption on X, dim X}---\ref{item: 0 in convex hull of x y z},
assume in addition that $Q \coloneqq \{x, y, z\}$ is a primitive collection. Then the corresponding primitive relation is
\begin{itemize}
    \item either $r(Q) \colon \, x+y+z = x_i+x_j$ for possibly equal $x_i, x_j \in P$,
    \item or $r(Q) \colon \,x+y+z=v$ for some $v \in G(\Sigma)$.
\end{itemize}
\end{Lem}

\begin{proof}
Since $X$ is Fano \ref{item: assumption on X, dim X}, the degree of the primitive relation $x+y+z = A$ is positive, so we can have only three possibilities for $A$. The first one with $A=0$ is actually not possible by our assumption \ref{item: m >= 3}. The second is $A=v$, and we cannot say much about $v$ at the moment. The last possibility is
\[r(Q)\colon\, x+y+z = u+v
\]
for some, possibly equal, $u,v \in G(\Sigma)$. 
By \cref{prop:deg1 extremal}, this is an extremal primitive relation, so by \cref{prop:extremal forms cones} applied to $r(Q)$ and $\tau=\{0\}$, we get that $\langle u,x,y\rangle$, $\langle u,y,z\rangle$, $\langle u,x,z\rangle \in \Sigma$.
By \cref{remark: plane geometry}, we may assume without loss of generality that $0 \in \Conv(\pi(u),\pi(x),\pi(y))$,
hence by \cref{lemma: x y z not a cone}, we get $u \in P$. The same argument applies to conclude $v\in P$.
\end{proof}

Hence we have two cases to consider: when 
$\deg(Q)$ is $1$ and when it is $2$.

\subsubsection{Degree one}

In this case, by \cref{lemma: PR of x y z}, we have \begin{enumerate}[start=\theassumption,label=(\Roman*)]
    \item\assumption \label{item: (a) PR of x y z is deg 1}
    a primitive relation $r(Q) \colon\, x+y+z = x_i + x_j$ for possibly equal $x_i,x_j \in P$.
\end{enumerate}

\begin{Lem}
\label{lemma: (a.1)}
In the setting \ref{item: assumption on X, dim X}---\ref{item: (a) PR of x y z is deg 1}, assume in addition that $G(\Sigma)$ is contained in $P \cup \{ x, y, z, x', y', z' \}$. Then $x'$, $y'$, $z'$ do not exist. Consequently, $\rho (X) \leq 2$.
\end{Lem}

\begin{proof}
By contradiction, assume that $x'\in G(\Sigma)$ exists, so $\{x,x'\}$ is a primitive collection, and let $x+x'=\alpha$ be the corresponding primitive relation. Clearly $\alpha \neq x$,~$x'$. The relation $r(Q)$ \ref{item: (a) PR of x y z is deg 1} gives 
\[ \alpha+y+z=x_i+x_j+x' , \]
which shows $\alpha \not \in P \cup \{y,z\}$ since otherwise the left hand side (LHS) forms a cone and we get an effective class of degree zero. Indeed, if $\alpha \in \{y,z\}$ it is clear that the LHS would form a cone; if $\alpha \in P$ applying \cref{prop:extremal forms cones} to $r(Q)$ and $\tau=\langle \alpha \rangle$ would follow that the LHS is a cone.

So without loss of generality, we can assume $x+x'=y'$. Applying the same argument to $y'$, we get that 
$y+y'=x'$ or $y+y'=z'$. The former would imply
$x+y=0$, which is not possible by \ref{item: m >= 3}, so $y+y'=z'$. Again, applying the same argument to $z'$, we get 
$z+z'=x'$. Summing the three primitive relations, we obtain $x+y+z=0$, which contradicts \ref{item: (a) PR of x y z is deg 1}.
\end{proof}

This leaves us with the case when
\begin{enumerate}[start=\theassumption,label=(\Roman*)]
    \item\assumption \label{item: (a.2) new vector u}
    there exists $u \in G(\Sigma) \setminus (P \cup \{ x,y,z, x',y',z' \})$.
\end{enumerate}
By \cref{remark: plane geometry}, we can assume without loss of generality that
\begin{enumerate}[start=\theassumption,label=(\Roman*)]
    \item\assumption \label{item: 0 in convex hull of x y u}
    $0 \in \Conv (\pi(x), \pi(y), \pi(u))$.
\end{enumerate}

\begin{Lem}
\label{lemma: PR of second PC of order 3}
Assume \ref{item: assumption on X, dim X}---\ref{item: 0 in convex hull of x y u}, then $R \coloneqq \{x,y,u\}$ is a primitive collection with 
primitive relation $r(R) \colon \,x+y+u = v$ for some $v\in G(\Sigma)$.
\end{Lem}

\begin{proof}
By \cref{lemma: x y z not a cone}, we have $\langle x,y,u \rangle \notin \Sigma$, and since $u \neq x', y'$, we conclude that $\{x,y,u\}$ is a primitive collection.
By \cref{lemma: PR of x y z}, we can have either
$r(R)\colon\,x+y+u = x_k + x_l$ or $r(R)\colon\,x+y+u = v$. In the former case, 
combining $r(R)$ with $r(Q)$ \ref{item: (a) PR of x y z is deg 1} provides us with a relation
\[ u + x_i + x_j = z + x_k + x_l.
\]
By applying \cref{prop:extremal forms cones} to $r(Q)$ \ref{item: (a) PR of x y z is deg 1} and $\tau=\langle x_k, x_l \rangle$, we get $\langle z,x_k,x_l\rangle \in \Sigma$ (here we are using the assumption $\dim(X)=n \geq 6$ \ref{item: assumption on X, dim X}).
Hence, by \cref{prop:effective classes}, we get an effective curve class of degree $0$, contradicting the Fano assumption \ref{item: assumption on X, dim X}.
\end{proof}

We will write down the result of \cref{lemma: PR of second PC of order 3} as an additional assumption, remembering that it is implied by the previous assumptions:
\begin{enumerate}[start=\theassumption,label=(\Roman*)]
    \item\assumption \label{item: PR of x y u}
    We have a primitive relation $r(R) \colon \,x+y+u = v$ for some $v\in G(\Sigma)$.
\end{enumerate}

\begin{Lem}
\label{lemma: (a.2) first approx of Sigma(1)}
Assume \ref{item: assumption on X, dim X}---\ref{item: PR of x y u}, then
\[G(\Sigma) \subset P \cup \{x,y,z, u, x',y',u',v'\} .\]
In particular, $z' \in  P \cup \{x,y,z, u, x',y',u',v'\}$, and since $v \neq x,y,u,v'$, we have that $v\in P$ or $v\in \{z,x',y',u'\}$.
\end{Lem}

\begin{proof}
Take any $w \in G(\Sigma) \setminus ( P \cup \{x,y,z, u, x',y',u'\})$.
By \cref{remark: plane geometry}, the convex hull of $\pi(w)$ together with two of $\pi(x)$, $\pi(y)$, $\pi(u)$ contains $0$, yielding an analog of \ref{item: 0 in convex hull of x y u}. By \cref{lemma: x y z not a cone} and from $w\neq x',y',u'$, it follows that one of 
$\{w,x,u\}$, $\{w,y,u\}$, $\{w,x,y\}$ is a primitive collection. 

We will prove that the corresponding primitive relation has the form
$w+x+u=b$ or $w+y+u=b$ or $w+x+y=b$ for some $b \in G(\Sigma)$. Assume for a contradiction that this is not the case.  By \cref{lemma: PR of x y z}, we have that 
one of
$w+x+u$, $w+y+u$, or $w+x+y$ equals $x_k+x_l$, for possibly equal $x_k, x_l\in P$. Hence
there exist $a,b\in G(\Sigma)$ (one of which is $w\neq z$) such that 
either $x+a+b$ or $y+a+b$ equals $x_k+x_l$. Assume $x+a+b=x_k+x_l$. 
Combining this with $r(Q)$ \ref{item: (a) PR of x y z is deg 1}, it follows that $y+z+x_k+x_l=a+b+x_i+x_j$. 
By applying \cref{prop:extremal forms cones} to $r(Q)$ \ref{item: (a) PR of x y z is deg 1} and $\tau=\langle x_k, x_l \rangle$, we get $\langle y,z,x_k,x_l\rangle \in \Sigma$ (here we are using the assumption $\dim(X)=n \geq 6$ \ref{item: assumption on X, dim X}). 
Hence, by \cref{prop:effective classes}, we get an effective curve class of degree $0$. The class is non-trivial since $w\neq y,z,x_k, x_l$. This contradicts the assumption that $X$ is Fano. 
The case $y+a+b=x_k+x_l$ is similar.

So we have a primitive relation of the form $w+x+u=b$ or $w+y+u=b$ or $w+x+y=b$ for some $b \in G(\Sigma)$. Combining this with $r(R)$ 
\ref{item: PR of x y u}, we get $b+y=w+v$ or $b+x=w+v$ or $b+u=w+v$. All possibilities imply that $w=v'$, as otherwise, by \cref{prop:effective classes}, we obtain an effective class of degree $0$ (non-trivial since $w,v\neq y,u$), which contradicts the Fano assumption \ref{item: assumption on X, dim X}.
\end{proof}

\begin{Lem}
\label{lemma: (a.2) second approx of Sigma(1)}
Assume \ref{item: assumption on X, dim X}---\ref{item: PR of x y u}.
Then $x$, $y$ don't have opponents,
\[ G(\Sigma) \subset P \cup \{x,y,z,u,u',v'\} , \]
and we have a trichotomy: $v \in P$ or $v = z$ or $v=u'$.
\end{Lem}

\begin{proof}
We will only show that $x'$ does not exist, and the argument for $y'$ is symmetric. Suppose to the contrary that $x'\in G(\Sigma)$, and let $x+x' = \alpha$ be the corresponding primitive relation, so $\alpha \neq x, x'$. Then we can substitute $x = \alpha - x'$ into $r(Q)$ \ref{item: (a) PR of x y z is deg 1} to get 
\[ \alpha +y+z = x'+x_i+x_j ,
\]
which shows $\alpha \notin P \cup \{y,z\}$, as otherwise we would get an effective curve class of degree $0$.
Substituting $x = \alpha - x'$ into $r(R)$ \ref{item: PR of x y u} gives the relation 
\begin{equation} \label{equ:lemma degree 1}
\alpha+y+u=v+x'.    
\end{equation}

Note that $\langle v, x' \rangle \in \Sigma$ by \cref{lemma: one true enemy}.
We claim that \cref{equ:lemma degree 1} is an extremal primitive relation by \cref{lemma: when is eff deg 1 PC}. Assume not, then \cref{lemma: when is eff deg 1 PC} implies that one of $\{\alpha, y, u\}$ is in $\{v,x'\}$ and the remaining two vectors are opponents. Since $u\neq y'$, then either $\alpha, y$ are opponents and $u\in\{v,x'\}$, or $\alpha, u$ are opponents and $y\in\{v,x'\}$.
From \ref{item: PR of x y u}, we have $u \neq v$ and $y \neq v$; \ref{item: (a.2) new vector u} means $u\neq x'$; and \ref{item: (a) PR of x y z is deg 1} implies $y\neq x'$. So both cases are impossible.

So \cref{equ:lemma degree 1} is an extremal primitive relation. 
In particular, $\alpha\neq y',u,u'$. By \cref{prop:extremal forms cones},
$v$ forms a cone with $\alpha$, hence $\alpha\neq v'$, which contradicts \cref{lemma: (a.2) first approx of Sigma(1)}.
\end{proof}

\begin{Lem}
\label{lemma: (a.2) third approx of Sigma(1)}
Assume \ref{item: assumption on X, dim X}---\ref{item: PR of x y u}. Then we have
\[ G(\Sigma) \subset P \cup \{x,y,z,u,v'\} \]
and a dichotomy:
$v \in P$ or $v = z$.
If moreover $u'$ exists, we have $u+u'=z$, $v=x_i$ and $u = x_j'$.
\end{Lem}

\begin{proof}
Assume that $u'$ exists and let $u+u'=\alpha$ be the corresponding primitive relation. Then substituting it into $r(R)$ \ref{item: PR of x y u} gives a relation
\begin{equation} \label{equ:lemma degree 1 b}
x+y+\alpha=u'+v.
\end{equation}
By $r(R)$ \ref{item: PR of x y u}, $v \neq u$, so $\langle u', v \rangle \in \Sigma$.
Since there are  no $x'$ and $y'$, it follows from 
\cref{lemma: when is eff deg 1 PC} that 
\cref{equ:lemma degree 1 b} is an extremal primitive relation.   
It follows that $\alpha \notin \{u,u',x,y, v'\}$. 
Moreover, $\alpha \notin P$, as otherwise applying \cref{prop:extremal forms cones} to $r(Q)$ and $\tau= \langle \alpha \rangle$ would imply $\langle x,y,\alpha \rangle \in \Sigma$. 
So by \cref{lemma: (a.2) second approx of Sigma(1)}, the only possibility is that $\alpha=z$. Therefore $u'+v=x_i+x_j$ by $r(Q)$ \ref{item: (a) PR of x y z is deg 1}, so we have, after possibly relabeling, that $v = x_i$, $u'=x_j$, which by \cref{lemma: (a.2) second approx of Sigma(1)} proves the statement.

If instead $u'$ does not exist, the statement follows directly from \cref{lemma: (a.2) second approx of Sigma(1)}.
\end{proof}

\begin{Lem}
\label{lemma: no enemy for z}
Assume \ref{item: assumption on X, dim X}---\ref{item: PR of x y u}.
Then $z$ doesn't have an opponent.
\end{Lem}

\begin{proof}
We argue by contradiction and assume that $z'$ exists.
Clearly $z' \neq x, y, z$ by $r(Q)$ \ref{item: (a) PR of x y z is deg 1} and
$z' \neq u$ by \ref{item: (a.2) new vector u}.
Furthermore, $z'\neq u'$, otherwise we have $z=u$ by \cref{lemma: one true enemy}, which contradicts \ref{item: (a.2) new vector u}.
Finally, $z' \notin P$, otherwise $z'=x_k$ implies $z=x_k'$, and $r(Q)$ \ref{item: (a) PR of x y z is deg 1} becomes
\[ x+y+x_k' = x_i+x_j,
\]
but applying \cref{prop:extremal forms cones} to $r(Q)$ and $\tau=\langle x_k \rangle$, we get $\langle x_k, x_k' \rangle \in \Sigma$, a contradiction.
Thus by \cref{lemma: (a.2) third approx of Sigma(1)}, the only possibility is $z'=v'$, so $z=v$ by \cref{lemma: one true enemy}. By \cref{lemma: (a.2) third approx of Sigma(1)}, this implies $G(\Sigma) \subset P \cup \{x,y,z,u,z'\}$.

Consider the primitive relation $z+z'=\beta \in G(\Sigma)$. We will show that $\beta = u$. Indeed, it is clear that $\beta \neq z,z'$. 
Combining $z+z'=\beta$ with $r(Q)$ \ref{item: (a) PR of x y z is deg 1}, we have 
\[x+y+\beta=x_i + x_j + z'.\]
If $\beta \in \{x,y\}$, the left hand side is a cone, and we get a non-trivial effective relation of degree 0, which is impossible by \ref{item: assumption on X, dim X}. If $\beta \in P$, applying \cref{prop:extremal forms cones} to $r(Q)$ and $\tau=\langle \beta \rangle$ implies that $\langle x,y,\beta \rangle \in \Sigma$, and we obtain a contradiction as before.

But now, substituting $z+z'=u$ into 
$r(R)$ 
\ref{item: PR of x y u} yields $x+y+z'=0$, hence contradicting $\minRC(X)\geq 3$ \ref{item: m >= 3}.
\end{proof}

\begin{Lem}
\label{lemma: (a.2) proved}
Assume \ref{item: assumption on X, dim X}---\ref{item: (a.2) new vector u}.
Then $\rho(X) \leq 3$. 
\end{Lem}

\begin{proof}
Let us summarize the consequences of \ref{item: assumption on X, dim X}---\ref{item: (a.2) new vector u}:
\begin{itemize}
    \item[{\ref{item: 0 in convex hull of x y u}}]
    $0 \in \Conv (\pi(x), \pi(y), \pi(u))$ after possibly relabeling $x$, $y$, $z$;
    \item[{\ref{item: PR of x y u}}] 
    we have a primitive relation $r(R)\colon x+y+u = v$ (\cref{lemma: PR of second PC of order 3});
    \item $x,y,z$ don't have opponents (\cref{lemma: (a.2) second approx of Sigma(1)}, \cref{lemma: no enemy for z});
    \item $G(\Sigma) \subset P \cup \{x,y,z,u,v'\}$ (\cref{lemma: (a.2) third approx of Sigma(1)}), so $\rho(X) \leq 4$;
    \item we have $v \in P$ or $v = z$ (\cref{lemma: (a.2) third approx of Sigma(1)}), so we can consider two cases.
\end{itemize}

\paragraph{Case $v=z$.} By \cref{lemma: no enemy for z}, $v'=z'$ does not exist, hence it follows from \cref{lemma: (a.2) third approx of Sigma(1)} that $G(\Sigma) \subset P \cup \{x,y,z,u\}$, so $\rho(X) \leq 3$.

\paragraph{Case $v\in P$, say $v=x_l$.} If $v=x_l$ doesn't have an opponent, then $G(\Sigma) \subset P \cup \{x,y,z,u\}$ and $\rho(X)\leq 3$.
So the tricky case is when $v=x_l$ has an opponent $x_l' \notin P$.

{Let $v+v'=\beta\in G(\Sigma)$. By $r(R)$, we have that $x+y+u+v'=v+v'=\beta$.
Since $\minRC(X)\geq3$, $\beta\neq x,y,u,v'$. Hence $\beta \in P\cup\{z\}$. If $\beta\in P$, let $\beta=x_k$. Then $v'=x_k-x_l$ and $k\neq l$. Using $r(P)$, we obtain that $v'=2x_k+\sum_{t\neq k,l}x_t$. As the vectors on the right hand side form a cone, we get an effective curve class of degree $2-n$, which is impossible by \ref{item: assumption on X, dim X}. 

So $\beta=z$. By $r(Q)$, we obtain that 
$x+y+v'=x_i+x_j-x_l$. If $l\neq i,j$, then again, by using $r(P)$ \ref{item: CSPR of length n-1}, we may write 
$x_i+x_j-x_l$ as  $x_i+x_j+ \sum_{t\neq l}x_t$.
As the vectors on the right hand side form a cone, we obtain an effective curve class of degree $3-n$, which is impossible by \ref{item: assumption on X, dim X}. 

So $l=i$ or $l=j$. Up to symmetry, we may assume $l=i$, so $v=x_i$, $x_i+x'_i=z$. By $r(Q)$, we have that
$x+y+x'_i=x_j$. Combining this with $r(R)$, we obtain that $x'_i+x_i=u+x_j$. Furthermore, $u\neq v'=x'_i$ and $u\neq x_i$.  
If $u,x_j$ form a cone, we obtain an effective non-trivial curve class of degree $0$, which is impossible by \ref{item: assumption on X, dim X}. 
So $u=x'_j$,
$v=x_i$ and we have two primitive relations}
\[ x_i+x_i'=z 
\quad\mbox{and}\quad
x_j+x_j'=z
. \]
Then either $i=j$, in which case $\rho(X) \leq 3$, or $i\neq j$, and we notice that they are both degree $1$ and hence extremal, so we can perform the contraction associated to one of them, say we contract the curve class $x_i+x_i'=z$:
\[\begin{tikzcd}[row sep=5pt]
	X & Y, \\
	V(z) & V(x_i,x_i').
	\arrow[from=1-1, to=1-2]
	\arrow[from=2-1, to=2-2]
\end{tikzcd}\]
By \cref{prop:PC_blowdown},
\begin{align*}
r(P_Y)\colon
& x_0 + \cdots + x_{n-2}=0, \\
r(R_Y)\colon 
& x+y+x'_j=x_i, \\
r(Q')\colon
& x+y+x'_i=x_j, \\
r(Q'')\colon
& x_j + x'_j=x_i + x'_i
\end{align*}
are primitive relations in $Y$.
Since $\rho(Y)=3$, we apply \cref{thm:Batyrev result rho=3}. 
We adopt the same notation as in \cref{thm:Batyrev result rho=3} and observe that we are in the case $l=5$, hence
\[ G(\Sigma)=\sqcup_{h=0}^{4}X_h=\{x_0, \dots, \check{x_j}, \dots, x_{n-2}\} \sqcup \{x_j\} \sqcup \{x,y\} \sqcup \{x'_j\} \sqcup \{x'_i\},\]
and either $X_2 \sqcup X_3=\{x_0, \dots, x_{n-2}\}$, or $\overline c= \overline b=\overline 0$ and $X_4 \sqcup X_0=\{x_0, \dots, x_{n-2}\}$. 
However, all possibilities for $\{x_j\}$ lead to a contradiction. Indeed:
\begin{itemize}
\item[-] if $X_3=\{x_j\}$, then we have $r_3\colon \, \overline 1\cdot X_3+\overline 1\cdot X_4= x_j+x'_j=x_i + x'_i \neq \overline 1\cdot X_h$ for all $h=0,\dots,4$;
\item[-] if $X_2=\{x_j\}$, then we have $r_1 \colon\, \overline 1\cdot X_1+\overline 1\cdot X_2= x'_j+x_j=x_i + x'_i \neq \overline 1\cdot X_h$ for all $h=0,\dots,4$;
\item[-] if $X_4=\{x_j\}$, then we have $r_3 \colon\, \overline 1\cdot X_3+\overline 1\cdot X_4= x'_j+x_j=x_i + x'_i \neq \overline 1\cdot X_h$ for all $h=0,\dots,4$;
\item[-] if $X_0=\{x_j\}$, then we have $r_0 \colon\, \overline 1\cdot X_0+\overline 1\cdot X_1= x_j+x'_j=x_i + x'_i \neq 
\cancel{\overline c \cdot X_2} 
+ (\cancel{\overline b} + \overline 1)X_3=\overline 1\cdot X_3$.
\end{itemize}
This concludes the last case, and we get that $\rho(X) \leq 3$.
\end{proof}

\begin{Prop}
\label{proposition: (a) x+y+z = xi+xj}
Let $X$ be a projective toric Fano manifold of $\dim(X)=n \geq 5$, $\minRC(X) \geq 3$, which admits a primitive relation
$r(P)\colon x_0+x_1+ \dots + x_{n-2}=0$ \ref{item: assumption on X, dim X}---\ref{item: CSPR of length n-1}.
Let $x,y,z \in G(\Sigma) \setminus P$ be such that $0 \in \Conv (\pi(x), \pi(y), \pi(z))$ \ref{item: x y z not in S}---\ref{item: 0 in convex hull of x y z}.
Assume in addition that $\{x, y, z\}$ is a primitive collection of degree $1$ \ref{item: (a) PR of x y z is deg 1}.
Then $\rho(X) \leq 3$.
\end{Prop}

\begin{proof}
We recall that assumptions \ref{item: assumption on X, dim X}---\ref{item: CSPR of length n-1} imply \ref{item: x y z not in S}---\ref{item: 0 in convex hull of x y z}. 
Then either $G(\Sigma) \subset P \cup \{ x, y, z, x', y', z' \}$, in which case $\rho(X) \leq 2$ by \cref{lemma: (a.1)}, or there exists a vector $u \in G(\Sigma) \setminus ( P \cup \{ x, y, z, x', y', z' \} )$ \ref{item: (a.2) new vector u}, in which case \cref{lemma: (a.2) proved} implies $\rho(X) \leq 3$, and we are done.
\end{proof}

\subsubsection{Degree two}

Still working in the setting 
\ref{item: assumption on X, dim X}---\ref{item: 0 in convex hull of x y z}, we assume that $\{x,y,z\}$ is a primitive collection whose primitive relation has degree $2$, and
by \cref{proposition: (a) x+y+z = xi+xj}, we can exclude the case when
degree one primitive collections as in 
\ref{item: x y z not in S}---\ref{item: (a) PR of x y z is deg 1} exist. 
In other words, these are the additional assumptions for this part of the proof:
\begin{enumerate}[start=\theassumption, label=(\Roman*)]
    \item\assumption
    \label{item: (b) primitive relation degree 2}
    we have a primitive relation $r(Q) \colon \,x+y+z = v$ for some $v \in G(\Sigma)$;
    \item\assumption
    \label{item: no prim rel of degree 1}
    there is no primitive collection $\{a,b,c\} \subset G(\Sigma) \setminus P$ with $0 \in \Conv(\pi(a),\pi(b),\pi(c))$ whose primitive relation has degree $1$.
\end{enumerate}

\begin{Lem}
\label{lemma: (b) first approximation of Sigma(1)}
In the setting \ref{item: assumption on X, dim X}---\ref{item: 0 in convex hull of x y z} and \ref{item: (b) primitive relation degree 2}---\ref{item: no prim rel of degree 1}, we have
\[  G(\Sigma) \subset P \cup \{ x, y, z, x', y', z', v'\}
\quad \text{and} \quad
v \in P \cup \{ x', y', z' \}.
\] 
\end{Lem}

\begin{proof}
If there is a generator $u \in G(\Sigma)\setminus (P \cup \{x,y,z,x',y',z'\})$, then by \cref{remark: plane geometry} and \ref{item: no prim rel of degree 1} we can assume that we have a primitive relation of the form 
\[x+y+u=w.\]
Combining it with \ref{item: (b) primitive relation degree 2}, we get $u+v=w+z$, so $u=v'$, otherwise $\langle u, v \rangle \in \Sigma$ and we get an effective non-trivial curve class of degree zero.

In particular, since $v \neq v'$, we have 
$v\in P \cup \{x',y',z'\}$.
\end{proof}

\begin{Lem}
\label{lemma: (b.1) there is no xm'}
In the setting \ref{item: assumption on X, dim X}---\ref{item: 0 in convex hull of x y z} and \ref{item: (b) primitive relation degree 2}---\ref{item: no prim rel of degree 1}, 
assume that $v=x_m\in P$. Then $x_m'$ does not exist.
\end{Lem}

\begin{proof}
The primitive relation $r(Q)$ \ref{item: (b) primitive relation degree 2} becomes $x+y+z=x_m$, and by \cref{lemma: (b) first approximation of Sigma(1)}, we have $G(\Sigma) \subset P \cup \{x,y,z,x',y',z', x_m'\}$.

Assume to the contrary that we have $x_m' \in G(\Sigma)$. 
Clearly, $x_m'\notin P$ because $x_m$ forms a $2$-dimensional cone in $\Sigma$ with any other generator $x_i\in P$. By \cref{remark: plane geometry}, $x_m'$ makes a primitive collection with two of $x,y,z$. 
Without loss of generality, assume we have a primitive relation of the form 
\[x+y+x_m'=w.\]
Combining it with $r(Q)$ \ref{item: (b) primitive relation degree 2}, we get $x_m'+x_m=z+w$, so $w=z'$ by \cref{lemma: sum of enemies gives more enemies}. 
Let $\alpha = x_m+x_m'=z+z'$. Then $\alpha \not\in P$ because $x_m' = \alpha - x_m$ is not in $\Span P$.
Now substituting $z = \alpha-z'$ into $r(Q)$ \ref{item: (b) primitive relation degree 2} yields
\begin{equation} \label{equ:lemma degree 2}
x+y+\alpha=x_m+z'.    
\end{equation}
Since $x_m \neq z$ by $r(Q)$ \ref{item: (b) primitive relation degree 2}, we have $\langle x_m, z'\rangle \in \Sigma$, so \cref{equ:lemma degree 2} is an extremal primitive relation. This implies that $\alpha \neq x,y,z,x',y',z',x'_m$, a contradiction with \cref{lemma: (b) first approximation of Sigma(1)}.
\end{proof}

\begin{Lem} \label{lemma: degree 2 bound 1}
In the setting \ref{item: assumption on X, dim X}---\ref{item: 0 in convex hull of x y z} and \ref{item: (b) primitive relation degree 2}---\ref{item: no prim rel of degree 1}, 
assume that $v=x_m\in P$. Then $\rho(X) \leq 3$.
\end{Lem}

\begin{proof}
By \cref{lemma: (b) first approximation of Sigma(1)} and \cref{lemma: (b.1) there is no xm'}, we have $G(\Sigma) \subset P \cup \{x,y,z,x',y',z'\}$, and, as before, the primitive relation \ref{item: (b) primitive relation degree 2} is $r(Q)\colon\,x+y+z=x_m$.

We show that at most one of $x',y',z'$ exists. Suppose to the contrary that for example $x',y' \in G(\Sigma)$, and let $\alpha =x+x'$, $\beta=y+y' \in G(\Sigma)$.
Since 
\[
\alpha+y+z=x_m+x'
\quad \text{and} \quad
\langle x_m, x' \rangle \in \Sigma,
\]
applying \cref{lemma: when is eff deg 1 PC} shows that 
this is an extremal primitive relation:
indeed, either $\alpha$ and $y$ are opponents and $z\in\{x_m,x'\}$, or $\alpha$ and $z$ are opponents and $y\in\{x_m,x'\}$. But $y,z\notin P\cup\{x'\}$, so this is a contradiction.
It follows that $\alpha \not\in\{x,x',y,y',z,z'\}$, so $\alpha =x_l \in P$; similarly, $\beta=x_k \in P$. But we have  
\[  \alpha+\beta+z=x'+y'+x_m , \]
which is not possible since we show that the right hand side forms a cone. 
Indeed, applying \cref{prop:extremal forms cones} to $x+x'=x_l$ and $\tau=\langle x_m, x_k \rangle$ (we use here that $n\geq5$), we obtain $\langle x_m, x_k, x' \rangle \in \Sigma$, and applying then \cref{prop:extremal forms cones} to  $y+y'=x_k$ and $\tau= \langle x_m, x' \rangle$ we have that $\langle x_m, x', y' \rangle \in \Sigma$.
So, after possibly relabelling $x$, $y$, $z$, we have $G(\Sigma) \subset P \cup \{x,y,z,x'\}$, hence $\rho(X) \leq 3$.
\end{proof}

\begin{Lem} \label{lemma: degree 2 bound 2}
In the setting \ref{item: assumption on X, dim X}---\ref{item: 0 in convex hull of x y z} and \ref{item: (b) primitive relation degree 2}---\ref{item: no prim rel of degree 1}, 
assume that $v = x'$. Then $\rho(X) \leq 3$.
\end{Lem}

\begin{proof}
We have $r(Q)\colon\,x+y+z=x'$. We know 
$G(\Sigma) \subset  P \cup \{x,y,z,x',y',z'\}$ by \cref{lemma: (b) first approximation of Sigma(1)}. So it is enough to show $y'$ and $z'$ do not exist. 
Suppose to the contrary that, for example,
$y' \in G(\Sigma)$, and let $\beta=y+y'$. Then 
\[\beta+x+z = x'+y'
\quad \text{and} \quad
\langle x', y' \rangle \in \Sigma
,\]
so again, by \cref{lemma: when is eff deg 1 PC}, this is an extremal primitive relation.
Indeed, assume the relation is not primitive. By \cref{lemma: when is eff deg 1 PC}, either $\beta$ and $x$ are opponents and $z\in\{x',y'\}$, or $\beta$ and $z$ are opponents and $x\in\{x',y'\}$.
But $z,x\notin\{x',y'\}$, since $\{x,y,z\}$ is a primitive collection. Hence $\beta+x+z =y'+x'$ is a primitive relation.
It follows from \cref{prop:extremal forms cones} that $\langle x, x' \rangle \in \Sigma$, which is a contradiction.
\end{proof}

\begin{Prop}
\label{proposition: (b) x+y+z = v}
Let $X$ be a projective toric Fano manifold of $\dim(X)=n \geq 5$, $\minRC(X) \geq 3$,
which admits a primitive relation
$r(P) \colon x_0+x_1+ \cdots + x_{n-2}=0$ \ref{item: assumption on X, dim X}---\ref{item: CSPR of length n-1}.
Assume that any primitive collection $\{x, y, z\}$ such that $x,y,z \in G(\Sigma) \setminus P$ and $0 \in \Conv (\pi(x), \pi(y), \pi(z))$ \ref{item: x y z not in S} ---\ref{item: 0 in convex hull of x y z} has degree $2$ \ref{item: no prim rel of degree 1}, and that there exists such a triple $x,y,z$ \ref{item: (b) primitive relation degree 2}.
Then $\rho(X) \leq 3$.
\end{Prop}

\begin{proof}
The statement follows from \cref{lemma: (b) first approximation of Sigma(1)}, \cref{lemma: degree 2 bound 1} and \cref{lemma: degree 2 bound 2}.
\end{proof}

\subsection{\texorpdfstring{Second case: none of $\{ x, y, z \}$}{\{x, y, z\}} form a primitive collection}
\label{subsection: none of x y z form a PC}

\begin{Prop} \label{prop: none PC}
Let $X$ be a toric Fano manifold of $\dim(X)=n \geq 5$, $\minRC(X) \geq 3$,
which admits a primitive relation
$r(P) \colon x_0+x_1+ \cdots + x_{n-2}=0$ \ref{item: assumption on X, dim X}---\ref{item: CSPR of length n-1}.
Assume in addition that none of the triples $\{x, y, z\} \subseteq G(\Sigma) \setminus P$ such that $0 \in \Conv (\pi(x), \pi(y), \pi(z))$ \ref{item: x y z not in S}---\ref{item: 0 in convex hull of x y z} form a primitive collection.
Then $\rho(X) \leq 3$.
\end{Prop}

Before proving \cref{prop: none PC}, we will formulate the following useful lemma.

\begin{Lem}
\label{lemma: none PC implies z is an opponent of x or y}
    In the setting of \cref{prop: none PC}, take any triple $\{x, y, z\} \subseteq G(\Sigma) \setminus P$ with $\langle x, y \rangle \in \Sigma$. Assume that $0 \in \Conv (\pi(x), \pi(y), \pi(z))$, or equivalently, $\pi(z) \in \langle -\pi(x), -\pi(y) \rangle$. Then $z = x'$ or $z = y'$.
\end{Lem}

\begin{proof}
By \cref{lemma: x y z not a cone}, $x$, $y$, $z$ do not span a cone, and by assumption, they do not form a primitive collection. So two of the vectors must not form a cone. Since we assumed that $\langle x, y \rangle \in \Sigma$, we must have $z = x'$ or $z = y'$.
\end{proof}
    
\begin{proof}[Proof of \cref{prop: none PC}]
We select $x,y \in G(\Sigma)$ with $\langle x,y\rangle \in \Sigma$ and such that the cone generated by $\pi(x)$ and $\pi(y)$ is maximal among cones in $\QQ^2 \simeq \QQ^n / \Gamma$ coming from such pairs. If there is $z\in G(\Sigma)$ such that $\pi(z)$ is outside the cone
$\langle \pi(x), \pi(y) \rangle$, then
we show that $z=x'$ or $z=y'$.
Indeed, the case $\pi(z) \in \langle -\pi(x), -\pi(y) \rangle$ is covered by \cref{lemma: none PC implies z is an opponent of x or y}; $\pi(z) \in \langle \pi(x) \rangle$ or $\pi(z) \in \langle \pi(y) \rangle$ cannot happen by an argument similar to \cref{lemma: x y z not a cone}; 
and in the remaining case, $\pi(z)$ is in $\langle \pi(x), -\pi(y) \rangle$ or $\langle -\pi(x),\pi(y)\rangle$, then we use maximality of the cone generated by $\pi(x)$ and $\pi(y)$.

Let $v$ be such that $\pi(v)$ is in the open half plane determined by $\Span \pi(x)$ not containing $\pi(y)$. Up to relabelling of $x$, $y$, we can assume that such a $v$ exists.
If $v=x'$, then $x+x'=y'$, since $\pi(x+x')$ is non-zero and is outside of $\langle \pi(y),\pi(x) \rangle.$ So 
\[ 
    \langle \pi(x), \pi(y) \rangle \subset
    \langle -\pi(y), -\pi(x') \rangle \cup
    \langle -\pi(x'), -\pi(y') \rangle \cup
    \langle -\pi(y'), -\pi(x) \rangle .
\]
It follows from 
\cref{lemma: none PC implies z is an opponent of x or y}
that $G(\Sigma) = P\cup\{x,y,x',y'\}$, which concludes this case. 

If now $v=y'$, in case $\pi(v) \in \langle -\pi(x), -\pi(y) \rangle$, we notice that $y+y'=x'$ as above, and in case $\pi(v) \in \langle \pi(x), -\pi(y) \rangle$, by completeness,
we find $w$ such that $\pi(w) \in \langle -\pi(x), \pi(y) \rangle \cup \langle -\pi(x), -\pi(y) \rangle$ and notice $w=x'$. In either case, we get
\[  \QQ^2 = 
    \langle -\pi(x), -\pi(y) \rangle \cup
    \langle -\pi(y), -\pi(x') \rangle \cup
    \langle -\pi(x'), -\pi(y') \rangle \cup
    \langle -\pi(y'), -\pi(x) \rangle ,
\]
and now \cref{lemma: none PC implies z is an opponent of x or y}
gives $G(\Sigma) = P\cup\{x,y,x',y'\}$.
\end{proof}

\section{Toric Fano manifolds with $\minRC(X)=n-2$}\label{sec:m=n-2}

In this section, we classify all $n$-dimensional toric Fano manifolds $X$ with $\minRC(X)=n-2$ and $n\geq 6$ (\cref{thm:classification_m>=n-2}). By \cref{theorem: P n-2 implies rho<=3}, we know that $\rho(X) \leq 3$.
\cref{decomposition-rank 2} and \cref{decomposition-rank 3} classify 
$n$-dimensional toric Fano manifolds $X$ with $\minRC(X)=n-2$, $n\geq 5$, and Picard rank $\rho(X)=2$ and $3$, respectively. 
Together, these results yield a classification of toric Fano manifolds with $\minRC(X)=n-2$ and $n \geq 6$.

\begin{Thrm}\label{decomposition-rank 2} 
Let $X$ be a toric Fano manifold of dimension $n \geq 5$, $\minRC(X) = n-2$ and $\rho(X) = 2$. Then $X\simeq\PS_{\PS^2}(\mathcal{E})$, where $\mathcal{E}$ is one of the following vector bundles on $\PS^2$:
\begin{itemize}
    \item $\mathcal{E}= \mathcal{O}_{\PS^2}(1) \oplus \mathcal{O}_{\PS^2}^{\oplus n-2} $,
    \item $\mathcal{E}= \mathcal{O}_{\PS^2}(1) \oplus \mathcal{O}_{\PS^2}(1) \oplus \mathcal{O}_{\PS^2}^{\oplus n-3} $,
    \item $\mathcal{E}= \mathcal{O}_{\PS^2}(2) \oplus \mathcal{O}_{\PS^2}^{\oplus n-2} $.
\end{itemize}
\end{Thrm}

\begin{proof}
Let $X$ be an $n$-dimensional toric Fano manifold with $\minRC(X)=n-2$ and $\rho(X)=2$.
We recall that toric manifolds with Picard rank $2$ are classified by \cite{Kleinschmidt1988}: they are projective space bundles over projective spaces.
The assumption $\minRC(X)=n-2$ implies that 
$X$ is a $\PS^{n-2}$-bundle over $\PS^2$.

We can write $X=\PS(\mathcal{E})$, where $\mathcal{E}= \mathcal{O}_{\PS^2}(a_0) \oplus \mathcal{O}_{\PS^2}(a_1)  \oplus \dots \oplus \mathcal{O}_{\PS^2}(a_{n-2})$ and $a_0 \geq a_1 \geq \dots \geq a_{n-2} = 0$. The Fano assumption on $X$ is equivalent to saying that $\sum_{i=0}^{n-2} a_i \leq 2$. 
Thus we have the following cases:
\begin{enumerate}
    \item $\mathcal{E}= \mathcal{O}_{\PS^2}^{\oplus n-1} $ and $X \simeq \PS^{n-2} \times \PS^2$,
    \item $\mathcal{E}= \mathcal{O}_{\PS^2}(1) \oplus \mathcal{O}_{\PS^2}^{\oplus n-2} $,
    \item  $\mathcal{E}= \mathcal{O}_{\PS^2}(1) \oplus \mathcal{O}_{\PS^2}(1) \oplus \mathcal{O}_{\PS^2}^{\oplus n-3} $,
    \item  $\mathcal{E}= \mathcal{O}_{\PS^2}(2) \oplus \mathcal{O}_{\PS^2}^{\oplus n-2} $.
\end{enumerate} 
In Case (1), we have $\minRC(X)=2$.
Cases (2)---(4) provide the complete list of toric Fano manifolds of Picard rank $2$ and $\minRC(X)=n-2$.
\end{proof}

\begin{Thrm}\label{decomposition-rank 3}
Let $X$ be a toric Fano manifold of dimension $n \geq 5$, $\minRC(X) = n-2$ and $\rho(X) = 3$. Then one of the following holds:        \begin{enumerate}
        \item  $X=\PS_S(\mathcal{E})$ is a $\PS^{n-2}$-bundle over a toric surface $S$, where $(S,\mathcal{E})$ is one of the following: 
                    \begin{itemize}
                     \item $S=\PS^1 \times \PS^1$ and $\mathcal{E}= \mathcal{O}_{\PS^1 \times \PS^1}(1,1) \oplus \mathcal{O}_{\PS^1 \times \PS^1}^{\oplus n-2} $,
                     \item $S=\PS^1 \times \PS^1$ and $\mathcal{E}= \mathcal{O}_{\PS^1 \times \PS^1}(1,0) \oplus \mathcal{O}_{\PS^1 \times \PS^1}(0,1) \oplus \mathcal{O}_{\PS^1 \times \PS^1}^{\oplus n-3} $,
                     \item $S=\F_1$ and $\mathcal{E}= \mathcal{O}_{\F_1}(e+f) \oplus \mathcal{O}_{\F_1}^{\oplus n-2}$, where $e\subset \F_1$ is the $-1$-curve, and $f\subset \F_1$ is a fiber of $\F_1\to \PS^1$.
                    \end{itemize} 
        \item Let $Y\simeq \PS_{\PS^2}\big(\mathcal{O}_{\PS^2}(1) \oplus \mathcal{O}_{\PS^2}^{\oplus n-2}\big)$ be the blowup of $\PP^n$ along a linear subspace $L=\PP^{n-3}$, and denote by $E\subset Y$ the exceptional divisor. 
        Then $X$ is the blowup of $Y$ along a codimension $2$ center $Z\subset Y$, where:
                \begin{itemize}
                     \item $Z$ is the intersection of $E$ with the strict transform of a hyperplane of $\PP^n$ containing the linear subspace $L$, or
                     \item $Z$ is the intersection of the strict transforms of two hyperplanes of $\PP^n$, one containing the linear subspace $L$, and the other one not containing it. 
                \end{itemize}
\end{enumerate}
\end{Thrm}

\begin{proof}
We apply Batyrev's classification of toric manifolds with $\rho(X)=3$, stated in \cref{thm:Batyrev result rho=3}. We adopt the same notation as in \cref{thm:Batyrev result rho=3}, and treat separately the cases when the number of primitive collections is $l=3$ and $l=5$. 

\medskip

\paragraph{Case $l=3$.} As $G(\Sigma)=X_0 \sqcup X_1 \sqcup X_2$ and $\minRC (X)=n-2$, we have $X_0=\{x_0, \dots, x_{n-2}\}$, $X_1=\{v_1, v_2\}$ and $X_2 =\{z_1, z_2\}$. The corresponding
primitive relations are all extremal by \cref{thm:Batyrev result rho=3}. 
By \cref{prop:Pk bundle}, $X$ is a $\PS^{n-2}$-bundle over a surface $S$. Up to relabelling in $X_0$, $X_1$ and $X_2$,
three possible choices for the remaining two primitive relations are (noting that $\minRC(X) = n-2 > 1$):
\[
\begin{cases}
v_1+v_2=x_0, \\
z_1+z_2=v_1;
\end{cases} \quad \hspace{12pt} \begin{cases}
v_1+v_2=x_0, \\
z_1+z_2=x_0;
\end{cases} \quad \quad \begin{cases}
v_1+v_2=x_0, \\
z_1+z_2=x_1.
\end{cases}
\]

It follows that $S$ is isomorphic to $\mathbb{F}_1$ when $r(X_1)$ and $r(X_2)$ are as in the first column above, or to $\PS^1 \times \PS^1$ in two other cases.

\medskip

\paragraph{Case $l=5$.} We denote by $l_i= |X_i| \geq 1$ the cardinality of $X_i$.

If $(\overline c,\overline b)=(\overline 0, \overline 0)$, then both $r_2 \colon \overline 1\cdot X_2 + \overline 1\cdot X_3=0$ and $r_4 \colon \overline 1\cdot X_4+\overline 1\cdot X_0=0$ are centrally symmetric primitive relations. 
By the assumption $\minRC(X)=n-2$,  we have 
$2n-2 \leq l_2+l_3+l_4+l_0  \leq |G(\Sigma)|-1=n+2$,
which implies $n \leq 4$, a contradiction.

So we have $(\overline c,\overline b)\neq (\overline 0,\overline 0)$, and $P \coloneqq X_2 \sqcup X_3$ is the only centrally symmetric primitive collection, so $l_2 + l_3 = n-1$ and $l_0 + l_1 + l_4 = 4$, with $l_0, l_1, l_4 \in \{1,2\}$.
As $\deg(r_0)>0$, we get the inequality
\[ 3 \geq 
l_0 + l_1 > 
\sum c_i + \sum b_j + l_3,
\]
which is only satisfied when $l_3 = 1$, $l_0+l_1 = 3$ and exactly one entry in \[(c_1, c_2, \dots, c_{l_2}, b_1)\] equals one, while the others are all zero. 
Up to relabelling, there are two cases: $c_1=1$ or $b_1=1$.

We then have $l_4 = 4 - (l_0 + l_1) = 1$.
From $\deg(r_3)>0$, we get
\[ l_3 + l_4 > l_1,
\]
which means $2>l_1$, and this ensures $l_1=1$, hence $(l_0,l_1,l_4) = (2,1,1)$.

To sum up, we denote $X_0=\{v_1, v_2\}$, $X_1=\{y\}$, $X_2 =\{x_1, \dots, x_{n-2}\}$, $X_3=\{x_{0}\}$ and $X_4=\{z\}$, and get the following two possibilities for $X$:
\begin{itemize}
    \item $b_1=1$ and $c_j=0$  for every $j$: 
    \begin{align*}
        r_0 &\colon\quad
            v_1 + v_2 + y = 2 x_{0},\\
        r_1 &\colon\quad
            y + x_1 + \cdots + x_{n-2} = z,\\
        r_2 &\colon\quad
            x_0 + \cdots + x_{n-2} = 0,\\
        r_3 &\colon\quad
            x_{0} + z = y,\\
        r_4 &\colon\quad
            z + v_1 + v_2 = x_{0}.
    \end{align*}
    \item $c_1=1$, $b_1=0$ and $c_j=0$ for every $j>1$:
    \begin{align*}
        r_0 &\colon\quad
            v_1 + v_2 + y = x_1 + x_{0},\\
        r_1 &\colon\quad
            y + x_1 + \cdots + x_{n-2} = z,\\
        r_2 &\colon\quad
            x_0 + \cdots + x_{n-2} = 0,\\
        r_3 &\colon\quad
            x_{0} + z = y,\\
        r_4 &\colon\quad
            z + v_1 + v_2 = x_1.
    \end{align*}
\end{itemize}

By \cref{prop:Pk bundle}, the open $U=X \setminus \big( V( y ) \cup V( z ) \big)$ has a $\PS^{n-2}$-bundle structure. The relation $r_3$ corresponds to an extremal curve class by \cref{prop:deg1 extremal}, which induces a  smooth blow-down $h\colon X \rightarrow Y$. By \cref{prop:PC_blowdown}, the primitive relations in $Y$ in cases $b_1=1$ and $c_1=1$  are, respectively:
\[
\begin{cases}
x_0 + \cdots + x_{n-2}=0, \\
z+v_1+v_2=x_0;
\end{cases}
\quad \quad \begin{cases}
x_0 + \cdots + x_{n-2}=0, \\
z+v_1+v_2=x_1.
\end{cases} 
\]
It follows that $Y\simeq \PS\left( \mathcal{O}_{\PS^2}(1) \oplus \mathcal{O}_{\PS^2}^{\oplus n-2} \right)$, i.e., $Y$ is the blowup of $\PP^n$ along a linear subspace $L=V(z,v_1,v_2)\simeq \PP^{n-3}\subset \PP^n$, with exceptional divisor $E=V(x_0)\subset Y$ in the first case, and $E=V(x_1)\subset Y$ in the second case.
The center of the blowup $h\colon X \rightarrow Y$ is $V(x_0, z)\subset Y$, yielding the two varieties described in (2).
\end{proof}
\smallskip

\appendix

\section{Code for computing primitive collections}\label{app}
\smallskip
\begin{center}by \textsc{Will Reynolds}\end{center}
\medskip

Let $\Sigma$ be the fan of a toric manifold of dimension $n$. 
For each integer $k\in\{1,\dots, n\}$, we denote by $\Sigma(k)$ the subset of $\Sigma$ consisting of the $k$-dimensional cones.
To determine whether a given subset $P\subseteq G(\Sigma)$ is a primitive collection, we proceed in two steps. 
First make sure there does not exist $\sigma\in\Sigma(n)$ with $P\subseteq G(\sigma),$ and if there does, stop; otherwise, for each $v\in P$, make sure there exist $\sigma\in\Sigma(n)$ with $P\setminus\{v\}\subseteq G(\sigma)$. 
In the special case $\lvert P \rvert = 2$ it suffices to check that there does not exist $\sigma\in\Sigma(n)$ with $P\subseteq G(\sigma)$.

Note that, by definition, if $P$ is a primitive collection and $P\subseteq Q,$ then $Q$ is not a primitive collection. Therefore, when looking for primitive collections, we go through subsets of $G(\Sigma)$ in increasing cardinality.

Assuming an implementation of the above basic algorithm, a reasonably efficient way to list all of the primitive collections of a fan is to arrive at such a list by eliminating $P\subseteq G(\Sigma)$ which are not primitive collections. The first step is to remove any $P$ with $\lvert P \rvert = 1.$ Then for each $P$ we check whether it is a primitive collection. If it is, we keep it and remove all sets containing it. If it is not, we remove it. One way of implementing this method of listing primitive collections is implemented in pseudocode in \cref{alg:PrimitiveCollections}.

This algorithm is impractical if $G(\Sigma)$ is too large. In practice, on a modern laptop, it works reasonably well up to about $\lvert{G(\Sigma)}\rvert=17,$ partly because of the combination of the following factors: first, eliminating all of the supersets of any $P$ with $\lvert P \rvert=2$ cuts down the remaining search space significantly, and second, the relative abundance of primitive collections of size 2, at least among toric Fano varieties. For example, the 124 toric Fano 4-folds altogether have 785 primitive collections, of which 566 have cardinality 2.

\thispagestyle{plain}

This last factor makes the computation of the value of $m(X)$ for a given toric Fano variety $X$ easier as well. Of the toric Fano varieties of a given dimension $n$ (for $n\leq6$) those $X$ with $m(X)=1$ make up an overwhelming majority. This means that computing $m(X)$ is usually extremely fast, even in the most straightforward way. The following table summarizes the data for $\dim(X)\in\{4,5,6\}.$

\begin{table}[ht]
\centering
\begin{tabular}{|l||l|l|l|l|l|l|l|}
\hline
\small{$\dim(X)$} & \small{\# Fanos}  & \small{\#(m=1)} & \small{\#(m=2)} & \small{\#(m=3)} & \small{\#(m=4)} & \small{\#(m=5)} & \small{\#(m=6)} \\ \hline \hline
4      & 124  & 107        & 15         & 1          & 1          &            &            \\ \hline
5      & 866  & 744        & 112        & 8          & 1          & 1          &            \\ \hline
6      & 7622 & 6333       & 1174       & 105        & 8          & 1          & 1         \\ \hline
\end{tabular}
\end{table}

\begin{algorithm}
\caption{List primitive collections of a given fan}\label{alg:PrimitiveCollections}
\textbf{Input:} fan $\Sigma$\;
$n \gets \dim{\Sigma}$\;
$PC \gets \{P\subseteq G(\Sigma):\lvert P \rvert>1\}$\;
\For{$P \in PC$ satisfying $\lvert P \rvert=2$}{
    \eIf{there exists $\sigma\in\Sigma(\text{n})$ such that $P\subseteq G(\sigma)$}{
        $PC \gets PC\setminus\{P\}$\;
        }{
        $PC \gets PC\setminus\{Q\in G(\Sigma):P\subsetneq Q\}$\;
        }
}
\For{$i\in\{3,\dots,n\}$}{
    \For{$P\in PC$ satisfying $\lvert P \rvert=i$}{
        \eIf{there exists $\sigma\in\Sigma(n)$ such that $P\subseteq G(\sigma)$}{
            $PC \gets PC\setminus\{P\}$\;
        }{
            $b \gets \textbf{True}$\;
            \For{$v \in P$}{
                \If{there does not exist $\sigma\in\Sigma(n)$ with $P\setminus\{v\}\subseteq G(\sigma)$}{
                    $b \gets \textbf{False}$\;
                } 
            }
            \eIf{$b$}{
                $PC \gets PC\setminus\{Q\subseteq G(\Sigma):P\subsetneq Q\}$\;
            }{
                $PC \gets PC\setminus\{P\}$\;
            }
        }
    }
}
\textbf{Output:} $PC$\;
\end{algorithm}

\newpage

For convenience, we also provide Macaulay2 code implementing the algorithm for computing primitive collections.

\thispagestyle{plain}

\begin{verbatim}
coneExistenceCheck = (S, fan) -> (
    for cone in fan do (
        if isSubset(S, cone) then (
            return true;
        );
    );
    return false;
);

properSubsetCheck = (S, fan) -> (
    for ray in S do (
        if coneExistenceCheck(S-set{ray}, fan) == false then (
            return false;
        );
    );
    return true;
);

isPrimitiveCollection = (P, Var) -> (
    if coneExistenceCheck(P, orbits(Var, 0)) then (
        return false)
    else (
        return properSubsetCheck(P, orbits(Var, 0)
    );
);

supsetsOfPrimColl = (E, B) -> (
    return set{for P in E-set{B} when isSubset(B, P) list P};
);

primitiveCollections = (Var) -> (
    n = length rays Var;
    primColls = select(subsets(toList(0..n-1)), x -> length x > 1);
    for P in subsets(toList(0..n-1), 2) do (
        if coneExistenceCheck(P, orbits(Var, 0)) == false then (
            primColls = primColls - supsetsOfPrimColl(primColls, P);)
        else (
            primColls = primColls - set{P};
        );
    );
    for i in toList(3..n) do (
        for P in subsets(toList(0..n-1), i) do (
            if member(P, primColls) == false then continue;
            if isPrimitiveCollection(P, Var) then (
                primColls = primColls - supsetsOfPrimColl(primColls, P);
            ) else (
                primColls = primColls - set{P};
            );
        );
    );
    return sort primColls;
);
\end{verbatim}
\thispagestyle{plain}

\smallskip

\printbibliography

@article {AraujoCastravet2012,
    AUTHOR = {Araujo, Carolina and Castravet, Ana-Maria},
     TITLE = {Polarized minimal families of rational curves and higher
              {F}ano manifolds},
   JOURNAL = {Amer. J. Math.},
  FJOURNAL = {American Journal of Mathematics},
    VOLUME = {134},
      YEAR = {2012},
    NUMBER = {1},
     PAGES = {87--107},
      ISSN = {0002-9327},
   MRCLASS = {14J45 (14C05 14J10)},
  MRNUMBER = {2876140},
MRREVIEWER = {Jin-Xing Cai},
       DOI = {10.1353/ajm.2012.0008},
       URL = {https://doi.org/10.1353/ajm.2012.0008},
}

@InCollection{AraujoCastravet2013,
  author     = {Araujo, Carolina and Castravet, Ana-Maria},
  title      = {Classification of 2-{F}ano manifolds with high index},
  booktitle  = {A celebration of algebraic geometry},
  year       = {2013},
  volume     = {18},
  series     = {Clay Math. Proc.},
  publisher  = {Amer. Math. Soc., Providence, RI},
  pages      = {1--36},
  mrclass    = {14J45 (14M20)},
  mrnumber   = {3114934},
  mrreviewer = {Zhiyu Tian},
}

@article {team2022,
    AUTHOR = {Araujo, Carolina and Beheshti, Roya and Castravet, Ana-Maria
              and Jabbusch, Kelly and Makarova, Svetlana and Mazzon, Enrica
              and Taylor, Libby and Viswanathan, Nivedita},
     TITLE = {Higher {F}ano manifolds},
   JOURNAL = {Rev. Un. Mat. Argentina},
  FJOURNAL = {Revista de la Uni\'{o}n Matem\'{a}tica Argentina},
    VOLUME = {64},
      YEAR = {2022},
    NUMBER = {1},
     PAGES = {103--125},
      ISSN = {0041-6932},
   MRCLASS = {14},
  MRNUMBER = {4477293},
       DOI = {10.33044/revuma.2921},
       URL = {https://doi.org/10.33044/revuma.2921},
}

@misc{BW,
  author = {Beheshti, Roya and Wormleighton, Ben},
  title = {Bounds on the Picard rank of toric Fano varieties with minimal curve constraints},
  publisher = {arXiv},
  note = {to appear in Proceedings of the AMS},
  year = {2022}
}

@article {Bat91,
    AUTHOR = {Batyrev, Victor V.},
     TITLE = {On the classification of smooth projective toric varieties},
   JOURNAL = {Tohoku Math. J. (2)},
  FJOURNAL = {The Tohoku Mathematical Journal. Second Series},
    VOLUME = {43},
      YEAR = {1991},
    NUMBER = {4},
     PAGES = {569--585},
      ISSN = {0040-8735},
   MRCLASS = {14M25 (52B20 52B35)},
  MRNUMBER = {1133869},
MRREVIEWER = {T. Oda}
}

@incollection {Bat99,
    AUTHOR = {Batyrev, V. V.},
     TITLE = {On the classification of toric {F}ano {$4$}-folds},
      NOTE = {Algebraic geometry, 9},
   JOURNAL = {J. Math. Sci. (New York)},
  FJOURNAL = {Journal of Mathematical Sciences (New York)},
    VOLUME = {94},
      YEAR = {1999},
    NUMBER = {1},
     PAGES = {1021--1050},
      ISSN = {1072-3374},
   MRCLASS = {14M25 (14J45)},
  MRNUMBER = {1703904},
MRREVIEWER = {Jaros\l aw A. Wi\'{s}niewski}
}

@article {Campana,
    AUTHOR = {Campana, F.},
     TITLE = {Connexit\'{e} rationnelle des vari\'{e}t\'{e}s de {F}ano},
   JOURNAL = {Ann. Sci. \'{E}cole Norm. Sup. (4)},
  FJOURNAL = {Annales Scientifiques de l'\'{E}cole Normale Sup\'{e}rieure. Quatri\`eme
              S\'{e}rie},
    VOLUME = {25},
      YEAR = {1992},
    NUMBER = {5},
     PAGES = {539--545},
      ISSN = {0012-9593},
   MRCLASS = {14J45},
  MRNUMBER = {1191735},
MRREVIEWER = {Luciana Picco Botta},
       URL = {http://www.numdam.org/item?id=ASENS_1992_4_25_5_539_0},
}

@article {Cas03b,
    AUTHOR = {Casagrande, Cinzia},
     TITLE = {Contractible classes in toric varieties},
   JOURNAL = {Math. Z.},
  FJOURNAL = {Mathematische Zeitschrift},
    VOLUME = {243},
      YEAR = {2003},
    NUMBER = {1},
     PAGES = {99--126},
      ISSN = {0025-5874},
   MRCLASS = {14M25 (14C25 14E30)},
  MRNUMBER = {1953051},
MRREVIEWER = {Evgeny Materov}
}

@article {Cas03,
    AUTHOR = {Casagrande, Cinzia},
     TITLE = {Toric {F}ano varieties and birational morphisms},
   JOURNAL = {Int. Math. Res. Not.},
  FJOURNAL = {International Mathematics Research Notices},
      YEAR = {2003},
    NUMBER = {27},
     PAGES = {1473--1505},
      ISSN = {1073-7928},
   MRCLASS = {14J25 (14E05 14M25)},
  MRNUMBER = {1976232},
MRREVIEWER = {Massimiliano Mella}
}

@article {Casagrande_Druel,
    AUTHOR = {Casagrande, Cinzia and Druel, St\'{e}phane},
     TITLE = {Locally unsplit families of rational curves of large
              anticanonical degree on {F}ano manifolds},
   JOURNAL = {Int. Math. Res. Not. IMRN},
  FJOURNAL = {International Mathematics Research Notices. IMRN},
      YEAR = {2015},
    NUMBER = {21},
     PAGES = {10756--10800},
      ISSN = {1073-7928},
   MRCLASS = {14J45 (14E30 14J40 14M22)},
  MRNUMBER = {3456027},
MRREVIEWER = {Andreas H\"{o}ring},
       DOI = {10.1093/imrn/rnv011},
       URL = {https://doi.org/10.1093/imrn/rnv011},
}

@inBook{Reid1983,
  author    = {Reid, Miles},
  booktitle = {Arithmetic and Geometry: Papers Dedicated to I.R. Shafarevich on the Occasion of His Sixtieth Birthday. Volume II: Geometry},
  date      = {1983},
  title     = {Decomposition of Toric Morphisms},
  editor    = {Artin, Michael and Tate, John},
  isbn      = {978-1-4757-9286-7},
  location  = {Boston, MA},
  pages     = {395--418},
  publisher = {Birkh{\"a}user Boston},
}

@article {CFH,
    AUTHOR = {Chen, Yifei and Fu, Baohua and Hwang, Jun-Muk},
     TITLE = {Minimal rational curves on complete toric manifolds and
              applications},
   JOURNAL = {Proc. Edinb. Math. Soc. (2)},
  FJOURNAL = {Proceedings of the Edinburgh Mathematical Society. Series II},
    VOLUME = {57},
      YEAR = {2014},
    NUMBER = {1},
     PAGES = {111--123},
      ISSN = {0013-0915},
   MRCLASS = {14M25 (14H10)},
  MRNUMBER = {3165015},
MRREVIEWER = {Nathan Owen Ilten}
}

@Book{CoxLittleSchenck2011,
  author     = {Cox, David A. and Little, John B. and Schenck, Henry K.},
  title      = {Toric varieties},
  year       = {2011},
  volume     = {124},
  series     = {Graduate Studies in Mathematics},
  publisher  = {American Mathematical Society, Providence, RI},
  isbn       = {978-0-8218-4819-7},
  pages      = {xxiv+841},
  mrclass    = {14M25 (05A15 05E45 52B12)},
  mrnumber   = {2810322},
  mrreviewer = {Ivan V. Arzhantsev},
}

@Article{dJS06,
  author       = {{de Jong}, A.~J. and {Starr}, Jason Michael},
  date         = {2006-02},
  journaltitle = {arXiv Mathematics e-prints},
  title        = {{A note on Fano manifolds whose second Chern character is positive}},
  eprint       = {math/0602644},
  eprintclass  = {math.AG},
  eprinttype   = {arXiv},
}

@article {deJongHeStarr11,
    AUTHOR = {{de Jong}, A. J. and He, Xuhua and Starr, Jason Michael},
     TITLE = {Families of rationally simply connected varieties over
              surfaces and torsors for semisimple groups},
   JOURNAL = {Publ. Math. Inst. Hautes \'{E}tudes Sci.},
  FJOURNAL = {Publications Math\'{e}matiques. Institut de Hautes \'{E}tudes
              Scientifiques},
    NUMBER = {114},
      YEAR = {2011},
     PAGES = {1--85},
      ISSN = {0073-8301},
   MRCLASS = {14D06 (14D23 14M22)},
  MRNUMBER = {2854858},
MRREVIEWER = {Alexandr V. Pukhlikov},
       DOI = {10.1007/s10240-011-0035-1},
       URL = {https://doi.org/10.1007/s10240-011-0035-1},
}

@misc{dJS06note_on_positive_ch2,
      title={A note on Fano manifolds whose second Chern character is positive},
      author={{de Jong}, Aise Johan and Starr, Jason Michael},
      year={2006},
      eprint={0602644},
      note = {arXiv:math/0602644},
      archivePrefix={arXiv:math},
      primaryClass={math.AG}
}

@misc{dJS06CI_are_RSC,
      title={Low degree complete intersections are rationally simply connected},
      author={{de Jong}, Aise Johan and Starr, Jason Michael},
      year={2006},
      url = {https://www.math.stonybrook.edu/~jstarr/papers/nk1006g.pdf}
}

@article {dJS07hF_and_rat_surfaces,
    AUTHOR = {{de Jong}, A. J. and Starr, Jason},
     TITLE = {Higher {F}ano manifolds and rational surfaces},
   JOURNAL = {Duke Math. J.},
  FJOURNAL = {Duke Mathematical Journal},
    VOLUME = {139},
      YEAR = {2007},
    NUMBER = {1},
     PAGES = {173--183},
      ISSN = {0012-7094},
   MRCLASS = {14J45 (14J10)},
  MRNUMBER = {2322679},
MRREVIEWER = {Alexandr V. Pukhlikov},
       DOI = {10.1215/S0012-7094-07-13914-0},
       URL = {https://doi.org/10.1215/S0012-7094-07-13914-0},
}

@book {Fulton,
    AUTHOR = {Fulton, William},
     TITLE = {Introduction to toric varieties},
    SERIES = {Annals of Mathematics Studies},
    VOLUME = {131},
      NOTE = {The William H. Roever Lectures in Geometry},
 PUBLISHER = {Princeton University Press, Princeton, NJ},
      YEAR = {1993},
     PAGES = {xii+157},
      ISBN = {0-691-00049-2},
   MRCLASS = {14M25 (14-02 14J30)},
  MRNUMBER = {1234037},
MRREVIEWER = {T. Oda},
       DOI = {10.1515/9781400882526},
       URL = {https://doi.org/10.1515/9781400882526},
}

@Article{Kleinschmidt1988,
  author       = {Kleinschmidt, Peter},
  date         = {1988},
  journaltitle = {aequationes mathematicae},
  title        = {A classification of toric varieties with few generators},
  issn         = {1420-8903},
  number       = {2},
  pages        = {254--266},
  volume       = {35},
  refid        = {Kleinschmidt1988},
}

@book {kollar96,
    AUTHOR = {Koll{\'a}r, J{\'a}nos},
     TITLE = {Rational curves on algebraic varieties},
    SERIES = {Ergebnisse der Mathematik und ihrer Grenzgebiete},
    VOLUME = {32},
 PUBLISHER = {Springer-Verlag},
   ADDRESS = {Berlin},
      YEAR = {1996},
}

@article {KMM92,
    AUTHOR = {Koll\'{a}r, J\'{a}nos and Miyaoka, Yoichi and Mori, Shigefumi},
     TITLE = {Rational connectedness and boundedness of {F}ano manifolds},
   JOURNAL = {J. Differential Geom.},
  FJOURNAL = {Journal of Differential Geometry},
    VOLUME = {36},
      YEAR = {1992},
    NUMBER = {3},
     PAGES = {765--779},
      ISSN = {0022-040X},
   MRCLASS = {14J45},
  MRNUMBER = {1189503},
MRREVIEWER = {Yuri G. Prokhorov},
       URL = {http://projecteuclid.org/euclid.jdg/1214453188},
}

@article {mori79,
    AUTHOR = {Mori, Shigefumi},
     TITLE = {Projective manifolds with ample tangent bundles},
   JOURNAL = {Ann. of Math. (2)},
  FJOURNAL = {Annals of Mathematics. Second Series},
    VOLUME = {110},
      YEAR = {1979},
    NUMBER = {3},
     PAGES = {593--606},
      ISSN = {0003-486X},
   MRCLASS = {14E05 (14M20)},
  MRNUMBER = {554387},
MRREVIEWER = {Luciana Picco Botta},
       DOI = {10.2307/1971241},
       URL = {https://doi.org/10.2307/1971241},
}

@article {Nagaoka,
    AUTHOR = {Nagaoka, Takahiro},
     TITLE = {On a sufficient condition for a {F}ano manifold to be covered
              by rational {$N$}-folds},
   JOURNAL = {J. Pure Appl. Algebra},
  FJOURNAL = {Journal of Pure and Applied Algebra},
    VOLUME = {223},
      YEAR = {2019},
    NUMBER = {11},
     PAGES = {4677--4688},
      ISSN = {0022-4049},
   MRCLASS = {14J45 (11B68)},
  MRNUMBER = {3955036},
MRREVIEWER = {Yu-Chao Tu},
       DOI = {10.1016/j.jpaa.2019.02.010},
       URL = {https://doi.org/10.1016/j.jpaa.2019.02.010},
}

@Article{Nobili2011,
  author       = {Nobili, Edilaine Ervilha},
  date         = {2011},
  journaltitle = {Bulletin of the Brazilian Mathematical Society, New Series},
  title        = {Classification of Toric 2-Fano 4-folds},
  issn         = {1678-7714},
  number       = {3},
  pages        = {399},
  volume       = {42},
  refid        = {Nobili2011},
}

@misc{Nobili2012,
  doi = {10.48550/ARXIV.1204.3883},
  url = {https://arxiv.org/abs/1204.3883},
   author = {Nobili, Edilaine Ervilha},
  title = {Birational Geometry of Toric Varieties},
  NOTE = {Thesis (Ph.D.)},
  publisher = {arXiv},
  year = {2012},
  copyright = {arXiv.org perpetual, non-exclusive license}
}

@article {Sato2012,
    AUTHOR = {Sato, Hiroshi},
     TITLE = {The numerical class of a surface on a toric manifold},
   JOURNAL = {Int. J. Math. Math. Sci.},
  FJOURNAL = {International Journal of Mathematics and Mathematical
              Sciences},
      YEAR = {2012},
     PAGES = {Art. ID 536475, 9},
      ISSN = {0161-1712},
   MRCLASS = {14M25 (14J45)},
  MRNUMBER = {2922100},
MRREVIEWER = {G. K. Sankaran},
       DOI = {10.1155/2012/536475},
       URL = {https://doi.org/10.1155/2012/536475},
}

@article {Sato2016,
    AUTHOR = {Sato, Hiroshi},
     TITLE = {Toric 2-{F}ano manifolds and extremal contractions},
   JOURNAL = {Proc. Japan Acad. Ser. A Math. Sci.},
  FJOURNAL = {Japan Academy. Proceedings. Series A. Mathematical Sciences},
    VOLUME = {92},
      YEAR = {2016},
    NUMBER = {10},
     PAGES = {121--124},
      ISSN = {0386-2194},
   MRCLASS = {14M25 (14E30 14J45)},
  MRNUMBER = {3579193},
MRREVIEWER = {Antonio Laface},
       DOI = {10.3792/pjaa.92.121},
       URL = {https://doi.org/10.3792/pjaa.92.121},
}

@Article{SanoSatoSuyama2020,
  author       = {Sano, Yuji and Sato, Hiroshi and Suyama, Yusuke},
  date         = {2021},
  journaltitle = {Kumamoto J. Math.},
  title        = {Toric {F}ano manifolds of dimension at most eight with positive second {C}hern characters},
  pages        = {1--13},
  volume       = {34},
  fjournal     = {Kumamoto Journal of Mathematics},
  mrclass      = {14M25 (14C17 14J45)},
  mrnumber     = {4240227},
}

@article {SatoSuyama2020,
    AUTHOR = {Sato, Hiroshi and Suyama, Yusuke},
     TITLE = {Remarks on toric manifolds whose {C}hern characters are
              positive},
   JOURNAL = {Comm. Algebra},
  FJOURNAL = {Communications in Algebra},
    VOLUME = {48},
      YEAR = {2020},
    NUMBER = {6},
     PAGES = {2528--2538},
      ISSN = {0092-7872},
   MRCLASS = {14M25 (14C17 14J45)},
  MRNUMBER = {4107588},
MRREVIEWER = {Alexey Shchuplev},
       DOI = {10.1080/00927872.2020.1719412},
       URL = {https://doi.org/10.1080/00927872.2020.1719412},
}

@Article{Sato2000,
  author       = {Sato, Hiroshi},
  date         = {2000},
  journaltitle = {Tohoku Math. J. (2)},
  title        = {Toward the classification of higher-dimensional toric {F}ano varieties},
  issn         = {0040-8735},
  number       = {3},
  pages        = {383--413},
  volume       = {52},
  fjournal     = {The Tohoku Mathematical Journal. Second Series},
  mrclass      = {14M25 (14J45)},
  mrnumber     = {1772804},
  mrreviewer   = {Sandra Di Rocco},
}

@book{Shrieve2020,
    AUTHOR = {Shrieve, Mike},
     TITLE = {On the {$\gamma_2$}-positivity of {S}mooth {T}oric
              {T}hreefolds},
      NOTE = {Thesis (Ph.D.)--University of Washington},
 PUBLISHER = {ProQuest LLC, Ann Arbor, MI},
      YEAR = {2020},
     PAGES = {54},
      ISBN = {979-8684-67090-9},
   MRCLASS = {Thesis},
  MRNUMBER = {4187487},
       URL =
              {http://gateway.proquest.com/openurl?url_ver=Z39.88-2004&rft_val_fmt=info:ofi/fmt:kev:mtx:dissertation&res_dat=xri:pqm&rft_dat=xri:pqdiss:28094976},
}

@misc{Starr06hypersurfacesRSC,
      title={Hypersurfaces of low degree are rationally simply-connected},
      author={Starr, Jason Michael},
      year={2006},
      eprint={0602641},
      %note = {arXiv:math/0602641},
      archivePrefix={arXiv:math},
      primaryClass={math.AG}
}

@article {Suzuki,
    AUTHOR = {Suzuki, Taku},
     TITLE = {Higher order minimal families of rational curves and {F}ano
              manifolds with nef {C}hern characters},
   JOURNAL = {J. Math. Soc. Japan},
  FJOURNAL = {Journal of the Mathematical Society of Japan},
    VOLUME = {73},
      YEAR = {2021},
    NUMBER = {3},
     PAGES = {949--964},
      ISSN = {0025-5645},
   MRCLASS = {14J45 (14C17 14M20)},
  MRNUMBER = {4291439},
       DOI = {10.2969/jmsj/84428442},
       URL = {https://doi.org/10.2969/jmsj/84428442},
}

\end{document}